\documentclass[letterpaper,onefignum,onetabnum]{siamart250211}
\usepackage{cleveref}
\usepackage{lipsum}
\usepackage{amsfonts}
\usepackage{amsmath}
\usepackage{amssymb}
\usepackage{graphicx}
\usepackage{epstopdf}
\usepackage{algorithmicx}
\usepackage{leftindex}
\usepackage{mathtools}
\usepackage{dsfont}
\usepackage{physics}        
\usepackage{fixdif}         
\usepackage{bm}             
\usepackage[T1]{fontenc}    
\usepackage{microtype}
\usepackage{amsopn}

\ifpdf
\DeclareGraphicsExtensions{.eps,.pdf,.png,.jpg}
\else
\DeclareGraphicsExtensions{.eps}
\fi

\newcommand{\ml}{\mathcal}
\newcommand{\md}{\mathds}
\newcommand{\mht}{\mathscr}
\newcommand{\mgt}{\mathfrak}
\newcommand{\ssy}{\md{R}}
\newcommand{\ssyn}{{\md{R}^n}}
\newcommand{\ssynn}{{\md{R}^{n\times n}}}

\newcommand{\zzs}{\md{N}_+}
\newcommand{\zrs}{\md{N}}
\newcommand{\gl}{\mathcal{P}}
\newcommand{\dkh}[1]{\left\{#1\right\}}
\newcommand{\xkh}[1]{\left(#1\right)}
\newcommand{\zkh}[1]{\left[#1\right]}

\newcommand{\Xc}{\bm{X}}

\newcommand{\xc}{\bm{x}}
\newcommand{\yc}{\bm{y}}
\newcommand{\zc}{\bm{z}}

\newcommand{\bc}{\bm{b}}
\newcommand{\tc}{\bm{t}}
\newcommand{\wc}{\bm{w}}
\newcommand{\mc}{\bm{m}}
\newcommand{\jc}{{\bm{j}}}
\newcommand{\dc}{{\bm{d}}}

\newcommand{\st}{\kg\mathrm{s.t.}\kg}
\newcommand{\fs}[1]{\left\lVert #1\right\rVert }

\newcommand{\jdz}[1]{\left\lvert #1\right\rvert }
\newcommand{\kg}{\;\;\!}
\newcommand{\id}{\mathrm{Id}}
\newcommand{\gee}{\geqslant}
\newcommand{\lee}{\leqslant}
\newcommand{\muc}{{\boldsymbol{\mu}}}
\newcommand{\nuc}{{\boldsymbol{\nu}}}
\newcommand{\lambdac}{{\boldsymbol{\lambda}}}

\newcommand{\gammac}{{\boldsymbol{\gamma}}}

\newcommand{\etac}{{\boldsymbol{\eta}}}
\newcommand{\tauc}{{\boldsymbol{\tau}}}
\newcommand{\deltac}{{\boldsymbol{\delta}}}
\newcommand{\thetac}{{\boldsymbol{\theta}}}
\newcommand{\ffss}{\ssy_{\gee0}}
\newcommand{\zdz}{{\mathbb{S}}_{++}}
\newcommand{\psd}{{\mathbb{S}}_{\gee0}}

\newcommand{\ffdjz}{{\mathbb{D}}_{\gee0}}
\newcommand{\zdjz}{{\mathbb{D}}_{>0}}
\newcommand{\dwz}{\mathrm{I}}

\newcommand{\vphi}{\varphi}
\newcommand{\kchf}[3]{\leftindex_{1}F_{1}(#1;#2;#3)}
\newcommand{\qw}{\mathbb{E}}

\DeclareMathOperator{\diag}{diag}
\DeclareMathOperator{\MW}{MW_{2}}
\DeclareMathOperator{\MWp}{MW}
\DeclareMathOperator{\W}{W_{2}}
\DeclareMathOperator{\Wp}{W}
\DeclareMathOperator{\MI}{MI}

\DeclareMathOperator*{\argmax}{arg\,max}
\DeclareMathOperator*{\argmin}{arg\,min}

\DeclareMathOperator*{\arginf}{arg\,inf}
\DeclareMathOperator{\law}{law}
\DeclareMathOperator{\Id}{Id}
\DeclareMathOperator{\proj}{proj}
\DeclareMathOperator{\Bary}{Bary}
\DeclareMathOperator{\onehot}{one\_hot}
\DeclareMathOperator{\EM}{EM}
\DeclareMathOperator{\Prob}{Prob}

\renewcommand{\email}[1]{\href{mailto:#1}{\nolinkurl{#1}}}

\allowdisplaybreaks[4]

\newcommand{\dy}{\coloneqq}
\newcommand{\zz}{\mathrm{T}}

\let\oldforall\forall
\DeclareRobustCommand{\forall}{\oldforall\,}

\let\oldexists\exists
\DeclareRobustCommand{\exists}{\oldexists\,}

\usepackage{xstring}
\usepackage{hyperref}
\hypersetup{
	colorlinks=true,
	allcolors=siaminlinkcolor
}

\newcommand{\bibcustomurl}[1]{%
	\IfBeginWith{#1}{https://doi.org}{%
		\StrSubstitute{#1}{https://doi.org/}{doi:}[\tempdoi]%
		\href{#1}{\nolinkurl{\tempdoi}}%
	}{%
		\IfBeginWith{#1}{https://arxiv.org}{%
			\StrSubstitute{#1}{https://arxiv.org/abs/}{arXiv:}[\temparxiv]%
			\StrSubstitute{\temparxiv}{https://arxiv.org/}{arXiv:}[\temparxiv]%
			\href{#1}{\nolinkurl{\temparxiv}}%
		}{%
			available at \href{#1}{\nolinkurl{#1}}%
		}%
	}%
}

\usepackage{tabulary}
\usepackage{booktabs}

\newcommand{\posdelta}[1]{\textcolor{green!60!black}{($\Delta$ + #1)}}
\newcommand{\negdelta}[1]{\textcolor{red}{($\Delta$ - #1)}}

\usepackage{pgfplots}
\usepackage{pgfplotstable}
\pgfplotsset{compat=1.18}

\usepackage{mathrsfs}

\theoremstyle{plain}
\newtheorem{assumption}{Assumption}
\Crefmultiformat{assumption}{Assumptions~#2#1#3}{ and~#2#1#3}{, #2#1#3}{ and~#2#1#3}
\Crefmultiformat{proposition}{Proposition~#2#1#3}{ and~#2#1#3}{, #2#1#3}{ and~#2#1#3}

\makeatletter
\newcommand{\subalign}[1]{%
	\vcenter{%
		\Let@ \restore@math@cr \default@tag
		\baselineskip\fontdimen10 \scriptfont\tw@
		\advance\baselineskip\fontdimen12 \scriptfont\tw@
		\lineskip\thr@@\fontdimen8 \scriptfont\thr@@
		\lineskiplimit\lineskip
		\ialign{\hfil$\m@th\scriptstyle##$&$\m@th\scriptstyle{}##$\hfil\crcr
			#1\crcr
		}%
	}%
}
\makeatother

\usepackage[noend]{algpseudocode}

\newsiamremark{remark}{Remark}
\crefname{remark}{Remark}{Remark}
\newsiamremark{hypothesis}{Hypothesis}
\crefname{hypothesis}{Hypothesis}{Hypotheses}
\newsiamthm{claim}{Claim}
\newsiamremark{fact}{Fact}
\crefname{fact}{Fact}{Facts}

\headers{LSMM-OTDA}{Songyan Luo and Yunxin Zhang}

\title{Minkowski-Type Wasserstein Metrics and Barycenters for Location-Scale Mixtures with Application to Domain Adaptation
}

\author{
	Songyan Luo
	\thanks{School of Mathematical Sciences, Fudan University, Shanghai 200433, China. (\email{syluo25@m.fudan.edu.cn}).}
	\and Yunxin Zhang
	\thanks{School of Mathematical Sciences, Shanghai Key Laboratory for Contemporary Applied Mathematics, Laboratory of Mathematics for Nonlinear Science, Fudan University, Shanghai 200433, China. (\email{xyz@fudan.edu.cn}).}
}

\begin{document}

\maketitle

% REQUIRED
\begin{abstract}
Discrete optimal transport (OT) typically relies on pointwise matching between empirical measures, incurring computational costs that scale at least quadratically with the sample size. To circumvent this limitation, we introduce a mathematical framework for OT between finite location-scale mixture models. By defining a specific function class grounded in generalized Minkowski inequalities and characterizing OT maps between multivariate location-scale families, we extend Wasserstein-type metrics and barycenters to these mixture models under the assumption of identifiability. Furthermore, we prove that restricting joint couplings to a specific mixture structure reduces the continuous multimarginal OT problem to a discrete transport problem over mixture components. Computing transport plans between these components rather than individual samples reduces the computational complexity to linear scaling with respect to the sample size. Empirical evaluations on the VisDA-C benchmark confirm that this strategy achieves competitive accuracy compared to existing empirical OT approaches, while substantially reducing the computational cost and memory footprint.
\end{abstract}

% REQUIRED

% REQUIRED

\section{Introduction}

Despite the widespread success of machine learning (ML) across various disciplines, traditional ML methodologies still exhibit inherent limitations when deployed in complex real-world scenarios. Achieving robust and reliable performance typically necessitates massive amounts of annotated training data, which incurs substantial computational and financial costs. While the standard independent and identically distributed (i.i.d.) assumption simplifies model training and theoretical analysis, the pervasive issue of domain shift~\cite{joaquinDatasetShiftMachine2008,ben-davidTheoryLearningDifferent2010} often degrades the generalization and robustness of models applied across disparate data distributions. To mitigate this issue, various transfer learning paradigms, ranging from domain adaptation (DA)~\cite{ganinDomainadversarialTrainingNeural2016} to domain generalization~\cite{liDeeperBroaderArtier2017}, have been proposed. Conventional DA focuses on transferring knowledge from a single labeled source domain to an unlabeled target domain. However, practical applications frequently aggregate labeled data from multiple heterogeneous sources, motivating the more complex paradigm of multi-source domain adaptation (MSDA)~\cite{pengMomentMatchingMultiSource2019}. Driven by this broader context, this paper establishes a theoretical framework whose underpinnings accommodate both two-marginal settings (i.e., standard DA between a single source and a target) and multimarginal settings (i.e., MSDA). Nevertheless, recognizing the distinct algorithmic complexities inherent in multi-source integration, our practical algorithmic design and empirical evaluations in this work are strictly focused on the foundational two-marginal DA scenario.

To address the distribution shift inherent in DA, optimal transport (OT) has emerged as a principled mathematical framework~\cite{courtyOptimalTransportDomain2017}. By defining robust metrics between probability distributions\,---\,most notably the Wasserstein and Gromov-Wasserstein distances~\cite{villaniOptimalTransport2009}\,---\,OT provides a geometrically meaningful way to align source and target domains. Compared to traditional information-theoretic measures (e.g., Kullback-Leibler divergence), OT-based metrics are particularly effective in capturing the underlying geometry of the data space. Recent methodologies have further expanded the theoretical and practical efficacy of OT. For instance, in the space of Gaussian mixture models (GMMs), Delon and Desolneux~\cite{delonWassersteinTypeDistanceSpace2020} introduced a Wasserstein-type distance that restricts the set of couplings, significantly alleviating computational burdens. Furthermore, in the context of MSDA, Wasserstein barycenters have been successfully employed to aggregate knowledge from multiple heterogeneous sources~\cite{montesumaWassersteinBarycenterMultisource2021}, highlighting the practical utility of barycentric formulations in transfer learning.

Despite its theoretical elegance, discrete OT typically relies on pointwise matching between empirical measures, incurring computational costs that scale at least quadratically with the sample size~\cite{peyreComputationalOptimalTransport2019}. This computational bottleneck restricts the applicability of empirical OT to large-scale datasets. To circumvent this limitation, representing data distributions via parametric or semi-parametric mixture models offers a viable alternative. A classical and widely adopted choice is the GMM~\cite{mclachlanFiniteMixtureModels2000}, whose origins trace back to Pearson. In this work, we theoretically consider a broader and unified class: the finite location-scale mixture model (LSMM)~\cite{fruhwirth-schnatterFiniteMixtureMarkov2006}. Serving as universal approximators for probability densities, LSMMs can capture complex, multimodal data distributions through a combination of simpler parametric components. This paradigm shift substantially reduces the representation complexity from the number of individual samples to the number of mixture components. While our theoretical framework is established for finite LSMMs, we specifically leverage GMMs in our algorithmic implementation and empirical evaluations to exploit their computational tractability.

The standard $p$-Wasserstein distance (associated with the cost function $c(\xc,\yc) = \|\xc-\yc\|^p$ for $p \gee 1$) serves as a cornerstone in OT, providing a profound geometric tool for comparing probability measures. Fundamentally, the metric properties of the $p$-Wasserstein distance are intimately connected to the classical Minkowski inequality. Recognizing that the Minkowski inequality can be generalized to accommodate a much broader class of functions~\cite{hardyInequalities1964}, a significant mathematical opportunity arises: establishing a generalized Wasserstein-type metric and its corresponding barycenter theory over finite LSMMs for this broader function class. In this article, we address this challenge by proposing a comprehensive mathematical framework. By defining a specific function class grounded in generalized Minkowski-type inequalities and characterizing the OT maps between multivariate location-scale families, we extend Wasserstein-type metrics to identifiable finite LSMMs. Then, we prove that by restricting joint couplings to a specific mixture structure, the continuous multimarginal OT problem reduces to a tractable discrete transport problem over mixture components.

The main contributions of this work are summarized as follows:
\begin{enumerate}
    \item \textbf{Theoretical Formulation of Metrics and Barycenters}: Grounded in generalized Minkowski-type inequalities, we introduce a specific function class to construct extended Wasserstein spaces, defining Wasserstein-type metrics and barycenters for identifiable finite LSMMs.
    \item \textbf{Continuous-to-Discrete Reduction and Bounding}: We prove that by restricting joint couplings to specific mixture structures, both two-marginal and multimarginal continuous OT problems between finite LSMMs can be reduced to tractable discrete OT problems. Furthermore, we establish inequalities bounding these discrete transport costs against their original continuous counterparts.
    \item \textbf{Characterization of Costs for Affine Optimal Maps}: We provide a complete characterization of translation-invariant cost functions that admit affine OT maps. Specifically, we prove that the optimality of symmetric positive definite affine maps requires the cost function to be separable, and that the presence of any cross-coupling further reduces the cost to a quadratic form.
    \item \textbf{Analytical and Closed-Form Formulas}: For general multivariate location-scale families, we derive analytical expressions for the generalized transport costs. When instantiating the framework with GMMs and specific functions from our proposed class, we provide closed-form formulas, expanding the analytical tools available for parametric OT.
    \item \textbf{Algorithmic Efficiency in DA}: Building upon these theoretical foundations, we propose a scalable domain adaptation framework instantiated via GMMs. By computing transport plans between mixture components rather than individual samples, the computational complexity is reduced to linear scaling with respect to the sample size. Empirical evaluations demonstrate significant reductions in computational cost and memory footprint while maintaining competitive accuracy compared to existing empirical approaches.
\end{enumerate}

The remainder of this article is organized as follows. Section 2 introduces preliminary concepts of Wasserstein spaces, barycenters, and OT for DA. Section 3 defines the specific function class and formulates the generalized metrics for finite mixture models. Section 4 establishes the reduction of continuous multimarginal OT problems to discrete OT problems. Section 5 characterizes translation-invariant cost functions admitting affine optimal transport maps and provides analytical or closed-form derivations for multivariate location-scale families. Section 6 details the algorithmic implementation of the proposed framework for domain adaptation and presents the experimental results. Finally, Section 7 concludes the paper.

\section{Preliminaries}\label{sec:Pre}
In this section, we will introduce several concepts related to the Wasserstein metrics and barycenters, which will be important for the following sections of this article. 
For further references, see~\cite{villaniOptimalTransport2009}.
Let $(\ml{X},d)$ be a Polish metric space. For a fixed Borel set $\Omega\subseteq\ml{X}$, we denote $\gl(\Omega)$ the collection of all probability measures on $\Omega$.

\subsection{Wasserstein spaces, metrics and barycenters}
For $p\gee1$, we denote by $\gl^d_p(\Omega)$ the set of probability measures on $\Omega$ such that 
\begin{equation*}
    \gl^d_p(\Omega)\dy\dkh{\muc\in\gl(\Omega):\exists\xc_0\in\Omega,\st\int_\Omega d(\xc,\xc_0)^p\d\muc(\xc)<+\infty},
\end{equation*}
where $\xc_0\in\Omega$ is arbitrary. By~\cite[Definition 6.4]{villaniOptimalTransport2009}, this space does not depend on the choice of the point $\xc_0$, and the Wasserstein metric $\Wp^d_p$ between $\muc_0,\muc_1\in\gl^d_p(\Omega)$ is defined by the formula
\begin{equation}\label{eq:Wp}
    \Wp^d_p(\muc_0,\muc_1)\dy\inf_{\gammac\in\prod(\muc_0,\muc_1)}\zkh{\int_{\Omega\times\Omega}d(\xc_0,\xc_1)^p\d\gammac(\xc_0,\xc_1)}^{1/p},
\end{equation}
where $\prod(\muc_0,\muc_1)$ denotes the set of joint probability distributions on $\Omega\times\Omega$ with marginals $\muc_0$ and $\muc_1$, also called the set of transport plans between $\muc_0$ and $\muc_1$. Now we denote the law of a random variable $\Xc$ by $\law(\Xc)$, the general lemma below guarantees the existence of a minimizer in~\eqref{eq:Wp}.
\begin{lemma}[Existence of an optimal coupling]\label{thm:existence of infimum}
    Given an integer $K\gee2$. For all $k\in\{0,\dots,K-1\}$, let $(\ml{X}_k,\muc_k)$ be a Polish probability space, and let $u_k:\ml{X}_k\to\ssy\cup\{-\infty\}$ be a upper semicontinuous function such that $u_k\in L^1(\muc_k)$. Let $c:\prod_{k=0}^{K-1}\ml{X}_k\to\ssy\cup\{+\infty\}$ be a lower semicontinuous cost function, such that $c(\xc_0,\dots,\xc_{K-1})\gee\sum_{k=0}^{K-1} u_k(\xc_k)$ for all $\xc_k\in\ml{X}_k,k\in\{0,\dots,K-1\}$. For any distribution $\gammac$ on $\prod_{k=0}^{K-1}\ml{X}_k$, let us define the cost of $\gammac$ with respect to $c$ by the formula
    \begin{equation*}
        C_c(\gammac)\dy\int_{\prod_{k=0}^{K-1}\ml{X}_k} c(\xc_0,\dots,\xc_{K-1})\d \gammac(\xc_0,\dots,\xc_{K-1}).
    \end{equation*}
    Let $\prod(\muc_0,\dots,\muc_{K-1})$ denotes the set of joint probability distributions on $\prod_{k=0}^{K-1}\ml{X}_k$ with marginals $(\muc_0,\dots,\muc_{K-1})$, and let $\ml{J}$ be a closed subset of $\prod(\muc_0,\dots,\muc_{K-1})$. Assume that the optimal cost of $(\muc_0,\dots,\muc_{K-1})$ with respect to $c$ and $\ml{J}$ is finite, i.e.
    \begin{equation*}
        C_c(\muc_0,\dots,\muc_{K-1};\ml{J})\dy\inf_{\gammac\in\ml{J}}C_c(\gammac)<+\infty.
    \end{equation*}
    Then there exists a random variable $\Xc_k$ on $\ml{X}_k$ for each $k\in\{0,\dots,K-1\}$, such that $\law(\Xc_k)=\muc_k$ and
    \begin{equation*}
        \law(\Xc_0,\dots,\Xc_{K-1})\in\arginf_{\gammac\in\ml{J}}C_c(\gammac).
    \end{equation*}
    Moreover, if $\ml{J}=\prod(\muc_0,\dots,\muc_{K-1})$, such joint probability distributions achieving the infimum are called optimal transport plans, and the couple $(\Xc_0,\dots,\Xc_{K-1})$ is called an optimal coupling of $(\muc_0,\dots,\muc_{K-1})$ with respect to $c$.
\end{lemma}

Although~\cref{thm:existence of infimum} is fundamental, a complete proof seems to be lacking in the literature. For completeness, we provide its proof below, which is built upon two auxiliary lemmas established below.

\begin{lemma}[Lower semicontinuity of the cost functional]\label{lem:Lower semicontinuity of the cost functional}
    Given an integer $K\gee2$. For all $k\in\{0,\dots,K-1\}$, let $\ml{X}_k$ be a Polish space, and let $c:\prod_{k=0}^{K-1}\ml{X}_k\to\ssy\cup\{+\infty\}$ be a lower semicontinuous cost function. Let $h:\prod_{k=0}^{K-1}\ml{X}_k\to\ssy\cup\{-\infty\}$ be an upper semicontinuous function such that $c\gee h$. Let $(\pi_\ell)_{\ell\in\zrs}$ be a sequence of probability measures on $\prod_{k=0}^{K-1}\ml{X}_k$, converging weakly to some $\pi\in\gl(\prod_{k=0}^{K-1}\ml{X}_k)$, in such a way that $h\in L^1(\pi_\ell)\cap L^1(\pi)$, and
    \begin{equation*}
        \int_{\prod_{k=0}^{K-1}\ml{X}_k}h\d\pi_\ell\xrightarrow[\ell\to+\infty]{}\int_{\prod_{k=0}^{K-1}\ml{X}_k}h\d\pi.
    \end{equation*}
    Then 
    \begin{equation*}
        \int_{\prod_{k=0}^{K-1}\ml{X}_k}c\d\pi\lee\liminf_{\ell\to+\infty}\int_{\prod_{k=0}^{K-1}\ml{X}_k}c\d\pi_\ell.
    \end{equation*}
    In particular, if $c$ is nonnegative, then $F:\pi\mapsto\int c\d\pi$ is lower semicontinuous on $\gl(\prod_{k=0}^{K-1}\ml{X}_k)$, equipped with the topology of weak convergence.
\end{lemma}
\begin{proof}
    The proof is analogous to that of~\cite[Lemma 4.3]{villaniOptimalTransport2009}. Replacing $c$ by $c-h$, we may assume that $c$ is a nonnegative lower semicontinuous function. Then $c$ can be written as the pointwise limit of a nondecreasing family $(c_{\ell'})_{\ell'\in\zrs}$ of continuous real-valued functions. By monotone convergence,
    \begin{align*}
        \int_{\prod_{k=0}^{K-1}\ml{X}_k}c\d\pi
        &=\lim_{\ell'\to+\infty}\int_{\prod_{k=0}^{K-1}\ml{X}_k}c_{\ell'}\d\pi \\
        &=\lim_{\ell'\to+\infty}\lim_{\ell\to+\infty}\int_{\prod_{k=0}^{K-1}\ml{X}_k}c_{\ell'}\d\pi_\ell
        \lee\liminf_{\ell\to+\infty}\int_{\prod_{k=0}^{K-1}\ml{X}_k}c\d\pi_\ell.
    \end{align*}
\end{proof}

\begin{lemma}[Tightness of transference plans]\label{lem:Tightness of transference plans}
    Given an integer $K\gee2$. For all $k\in\{0,\dots,K-1\}$, let $\ml{X}_k$ be a Polish space, and let $\ml{T}_k$ be a tight subset of $\gl(\ml{X}_k)$. Then the set $\prod_{k=0}^{K-1}\ml{T}_k$ of all transference plans with marginals $(\ml{T}_0,\dots,\ml{T}_{K-1})$, is itself tight in $\gl(\prod_{k=0}^{K-1}\ml{X}_k)$.
\end{lemma}
\begin{proof}
    The proof of this multivariate lemma proceeds similarly to the bivariate case in~\cite[Lemma 4.4]{villaniOptimalTransport2009}. For all $k\in\{0,\dots,K-1\}$, let $\muc_k\in\ml{T}_k$, and let $\pi\in\Pi(\muc_0,\dots,\muc_{K-1})$. For any $\varepsilon>0$, there is a compact set $K_{\varepsilon,k}\subseteq\ml{X}_k$, independent of the choice of $\muc_k$ in $\ml{T}_k$, such that $\muc_k[\ml{X}_k\backslash K_{\varepsilon,k}]<\varepsilon$. Then for any coupling $(X_0,\dots,X_{K-1})$ of $(\muc_0,\dots,\muc_{K-1})$,
    \begin{equation*}
        \Prob[(X_0,\dots,X_{K-1})\notin K_{\varepsilon,0}\times\dots\times K_{\varepsilon,K-1}]\lee\sum_{k=0}^{K-1}\Prob[X_k\notin K_{\varepsilon,k}]<K\varepsilon.
    \end{equation*}
    Moreover, since $\prod_{k=0}^{K-1}K_{\varepsilon,k}$ is compact in $\prod_{k=0}^{K-1}\ml{X}_k$, the proof is now complete.
\end{proof}

\begin{proof}[Proof of~\cref{thm:existence of infimum}]
    Our proof extends the method of~\cite[Theorem 4.1]{villaniOptimalTransport2009} to our generalized setting. For all $k\in\{0,\dots,K-1\}$, by Prokhorov's theorem~\cite{prokhorovConvergenceRandomProcesses1956}, since $\ml{X}_k$ is Polish, $\{\muc_k\}$ is tight in $\gl(\ml{X}_k)$. By~\cref{lem:Tightness of transference plans} and the definition of tightness, $\ml{J}$ is tight in $\gl(\prod_{k=0}^{K-1}\ml{X}_k)$, and by Prokhorov’s theorem this set has a compact closure. By passing to the limit in the equation for marginals, we see that $\prod(\muc_0,\dots,\muc_{K-1})$ is closed, so $\ml{J}$ is in fact compact.

    Then let $(\pi_\ell)_{\ell\in\zrs}$ be a sequence of probability measures on $\prod_{k=0}^{K-1}\ml{X}_k$, such that $\int c\d \pi_\ell$ converges to the infimum transport cost. Extracting a subsequence if necessary, we may assume that $\pi_\ell$ converges to some $\pi\in\prod(\muc_0,\dots,\muc_{K-1})$. The function $h:(\xc_0,\dots,\xc_{K-1})\mapsto\sum_{k=0}^{K-1} u_k(\xc_k)$ lies in $L^1(\pi_\ell)\cap L^1(\pi)$, so~\cref{lem:Lower semicontinuity of the cost functional} implies
    \begin{equation*}
        \int_{\prod_{k=0}^{K-1}\ml{X}_k}c\d\pi\lee\liminf_{\ell\to+\infty}\int_{\prod_{k=0}^{K-1}\ml{X}_k}c\d\pi_\ell.
    \end{equation*}
    Thus $\pi$ is minimizing, this yields the desired result.
\end{proof}

Moreover, from~\cite[Theorem 10.28]{villaniOptimalTransport2009}, there is a unique (in law) optimal coupling $(\Xc_0,\Xc_1)$ achieving the infimum in~\eqref{eq:Wp} on some assumptions, and there exists a measurable function $T:\Omega\to\Omega$ such that $\Xc_1 = T(\Xc_0)$. Besides, the function $T$ is called an optimal transport map from $\muc_0$ to $\muc_1$ and satisfying $T_\#\muc_0=\muc_1$. Here, we denote by $T_\#\muc_0$ the push-forward measure of $\muc_0$ by $T$, that is a Borel measure on $\Omega$, defined by
\begin{equation*}
    (T_\#\muc_0)[A]=\muc_0[T^{-1}(A)],\quad\forall A\subset\Omega.
\end{equation*}
Also, the optimal transport plan $\gammac^\ast$ has the following form
\begin{equation*}
    \gammac^\ast=(\id,T)_\#\muc_0.
\end{equation*}

Throughout this article, for the important particular case where the metric $d$ is the Euclidean distance $\fs{\cdot}_2$, we always drop the superscript $^d$ in the notation for the ease of readability, such as $\Wp_p\dy\Wp^{\fs{\cdot}_2}_p$. In this case, let $n\in\zzs$, and let $\ml{X}=\ssyn$, if $\muc_0$ is absolutely continuous with respect to the Lebesgue measure $\mathcal{L}^n$, then there is a unique (in law) optimal coupling achieving the infimum in~\eqref{eq:Wp}~\cite[Theorem 9.4]{villaniOptimalTransport2009}.

We next introduce the Wasserstein barycenters in the metric space $(\gl^d_p(\Omega),\Wp^d_p)$ which was introduced in~\cite{legouicExistenceConsistencyWasserstein2017}. Let $K\in\zzs$ and let
\begin{align*}
    \Delta^K\dy&\dkh{\tauc\dy(\tau_0,\dots,\tau_{K-1})\in[0,1]^K:\sum_{k=0}^{K-1}\tau_k=1}, \\
    \Delta^K_+\dy&\dkh{\tauc\dy(\tau_0,\dots,\tau_{K-1})\in(0,1]^K:\sum_{k=0}^{K-1}\tau_k=1}.
\end{align*}
Assume that $(\Omega,d)$ is a separable locally compact geodesic space. Given the family of measures $(\muc_k)_{0\lee k\lee K-1}\in\gl_p^d(\Omega)^K$ and weights $\tauc\dy(\tau_k)_{0\lee k\lee K-1}\in\Delta^{K}$, there exists a minimizer of the problem:
\begin{equation*}
    \Bary_\tauc(\muc_0,\dots,\muc_{K-1})\in\arginf_{\nu\in\gl^d_p(\Omega)}\sum_{k=0}^{K-1}\tau_k\Wp^d_p(\nuc,\muc_k)^p,
\end{equation*}
which is the barycenter of the family of measures $(\muc_k)_{0\lee k\lee K-1}$ with barycentric weights $\tauc$.

\subsection{Optimal transport between Gaussian distributions}
Let $n\in\zzs$, and let $\ml{X}=\Omega=\ssyn$. For each $i\in\{0,1\}$, let $\muc_i=\ml{N}(\mc_i,\Sigma_i)$ be a Gaussian distribution with mean $\mc_i\in\ssyn$ and covariance matrix $\Sigma_i\in\psd^n$. Here, we denote by $\psd^n$ the set of positive semi-definite matrices (PSD) in $\ssy^{n\times n}$, and by $\zdz^n$ the set of symmetric positive definite matrices (SPD) in $\ssy^{n\times n}$. Specifically, for any $\mc\in\ssyn$ and $\Sigma\in\psd^n$, we denote the probability density function of $\ml{N}(\mc,\Sigma)$ by
\begin{equation*}
    g_{\mc,\Sigma}(\xc)\dy (2\pi)^{-n/2}|\Sigma|^{-n/2}e^{(\xc-\mc)^\zz\Sigma^{-1}(\xc-\mc)},\quad\forall\xc\in\ssyn.
\end{equation*}
From~\cite{takatsuWassersteinGeometryGaussian2011}, the 2-Wasserstein distance $\W$ between $\muc_0$ and $\muc_1$ is given by
\begin{equation*}
    \W(\muc_0,\muc_1)^2=\fs{\mc_0-\mc_1}_2+\Tr{\Sigma_0+\Sigma_1-2(\Sigma_0^{\frac{1}{2}}\Sigma_1\Sigma_0^{\frac{1}{2}})^{\frac{1}{2}}},
\end{equation*}
where, for any $M\in\psd^n$, we denote by $M^{\frac{1}{2}}$ its unique positive semi-definite square root. If $\Sigma_0$ is non-singular, then the optimal transport map $T$ from $\muc_0$ to $\muc_1$ has a closed and affine form
\begin{equation*}
    T(\xc)=\Sigma_0^{-\frac{1}{2}}(\Sigma_0^{\frac{1}{2}}\Sigma_1\Sigma_0^{\frac{1}{2}})^{\frac{1}{2}}\Sigma_0^{-\frac{1}{2}}(\xc-\mc_0)+\mc_1=\Sigma_0^{-1}(\Sigma_0\Sigma_1)^{\frac{1}{2}}(\xc-\mc_0)+\mc_1.
\end{equation*}

\subsection{Optimal transport for domain adaptation}\label{sec:Optimal transport for domain adaptation}
In this paper, we focus on DA for classification tasks. Given $n,L,K\in\zzs$, and $n_0,\dots,n_{K-1},n_T\in\zzs$. As an important case, we consider an adaptation problem with $K$ labeled source domains and a single unlabeled target domain. Let $X_f(\subseteq\ssyn)$ be the feature space, $Y_\ell\dy\Delta^L$ be the label space with one-hot encoding, $f^\ast:X_f\to \Delta^{L}$ be the target function to learn, and $F_\ell:\Delta^{L}\times\Delta^L\to\ssy$ be a loss function penalizing errors with respect to $f^\ast$. The input is $K$ labeled source datasets $\ml{D}_0^S\dy\{(\xc_{0,i},\yc_{0,i})\}_{i=1}^{n_0},\dots,\ml{D}_{K-1}^S\dy\{(\xc_{K-1,i},\yc_{K-1,i})\}_{i=1}^{n_{K-1}}$ and an unlabeled target dataset $\ml{D}^T\dy\{\xc_{T,i}\}_{i=1}^{n_T}$. Here, we consider the empirical risk minimization principle~\cite{vapnikNatureStatisticalLearning2000}, which suggests minimizing a functional similar to the following form:
\begin{equation*}
    R(f)\dy\frac{1}{\sum_{k=0}^{K-1}n_k}\sum_{k=0}^{K-1}\sum_{i=1}^{n_k}F_\ell(f(\xc_{k,i}),\yc_{k,i}).
\end{equation*}

The concept of optimal transport for domain adaptation was introduced by~\cite{courtyOptimalTransportDomain2017}, focusing on aligning samples from $\ml{D}^S\dy\bigcup_{0\lee k\lee K-1}\ml{D}_k^S$ to those of $\ml{D}^T$ by solving an empirical optimal transport problem. Once the optimal map $T^\ast$ is found, this approach effectively generates a new dataset $\{(T^\ast(\xc),\yc):(\xc,\yc)\in\ml{D}^S\}$ using the barycentric map, where the source domain points retain their original labels after being mapped to the target domain. This procedure is valid under the covariate shift assumption~\cite{sugiyamaCovariateShiftAdaptation2007}, but it has cubic computational complexity with respect to the number of samples. In this work, we address this limitation by generalizing the GMM-OT framework proposed by~\cite{delonWassersteinTypeDistanceSpace2020,montesumaOptimalTransportDomain2025}.

\section{Metrics between finite mixtures}\label{sec:Distance between generic mixtures}
This section establishes metrics for finite mixtures. Our main result (\cref{thm:mixture metric}) guarantees the well-definedness of these extended metrics via generalized Minkowski-type inequalities.
First, let us define a nonempty set of particular distributions, which we call atoms, and mixtures thereof. These notation is similar to that used in~\cite{dussonWassersteintypeMetricGeneric2026}.
\begin{definition}[$\ml{A}$-mixture]
    Let $\ml{A}$ be a subset of $\gl(\Omega)$, called dictionary of atoms. We denote by $\ml{M}(\ml{A})$ the set of finite mixtures of $\ml{A}$, i.e. $\nuc\in\ml{M}(\ml{A})$ if and only if there exist $J\in\zzs$, $(\muc_1,\dots,\muc_J)\in\ml{A}^J$ and $(\lambda_1,\dots,\lambda_J)\in\Delta^J$ such that 
    \begin{equation*}
        \nuc(\xc)=\sum_{j=1}^J \lambda_j\muc_j,
    \end{equation*}
    which means
    \begin{equation*}
        F_{\nuc}(\xc)=\sum_{j=1}^J \lambda_jF_{\muc_j}(\xc),\quad\forall\xc\in\Omega,
    \end{equation*}
    where $F_{\nuc},F_{\muc_1},\dots,F_{\muc_J}$ are corresponding cumulative distribution functions. Then $\ml{M}(\ml{A})\subseteq\gl(\Omega)$.
\end{definition}

To define a metric on the set of finite mixtures and derive essential properties, the following identification assumption is required. 
\begin{assumption}[Identifiability]\label{ass:Identifiability}
    We say that $\ml{M}(\ml{A})$ is identifiable, if for any two mixtures of $\ml{A}$,
    \begin{equation*}
        \nuc_i=\sum_{j_i=1}^{J_i} \lambda_{i,j_i}\muc_{i,j_i},\quad \lambda_{i,j_i}>0,\kg\forall j_i\in\{1,\dots,J_i\}, i\in\{0,1\}, 
    \end{equation*}
    where the atomic distributions $(\muc_{0,1},\dots,\muc_{0,J_0})$ are pairwise distinct, and similarly for $(\muc_{1,1},\dots,\muc_{1,J_1})$. Then $\nuc_0=\nuc_1$ if and only if $J_0=J_1$, and the indices in the sums can be reordered such that $\lambda_{0,j}=\lambda_{1,j}$ and $\muc_{0,j}=\muc_{1,j}$ for all $j=1,\dots,J_0$.
\end{assumption}

Classical examples of identifiable mixtures are the set of Gaussian or Cauchy mixtures~\cite{yakowitzIdentifiabilityFiniteMixtures1968}, more identifiable mixtures can be found in~\cite{holzmannIdentifiabilityFiniteMixtures2006}. To define metrics between mixtures, we begin with the following metric assumption on the dictionary of atoms.

\begin{assumption}\label{ass:metric}
    The function $W:\ml{A}\times\ml{A}\to\ffss$ defines a finite metric on $\ml{A}$.
\end{assumption}

In particular, the Wasserstein metric $\Wp^d_p$ on $\gl^d_p(\Omega)$ satisfies~\cref{ass:metric} with $p\gee1$~\cite{villaniOptimalTransport2009}. In order to get more metrics that satisfy~\cref{ass:metric}, we introduce the following collection of functions.

\begin{definition}[The collection of functions, $\MI$~\cite{matkowskiConverseMinkowskisInequality1990}]
    For any increasing bijection $\varphi:\ffss\to\ffss$, we can define
    \begin{equation*}
        F_\vphi:\ffss^2\to\ffss,(x_0,x_1)\mapsto\vphi(\vphi^{-1}(x_0)+\vphi^{-1}(x_1)),\quad\forall (x_0,x_1)\in\ffss^2.
    \end{equation*}
    We denote by $\MI$ the set of all the increasing bijection $\varphi:\ffss\to\ffss$ such that $F_\vphi$ is concave in $\ffss^2$.
\end{definition}

The functions in $\MI$ satisfy the following inequality, which is a generalization of the Minkowski inequality introduced by Hardy~\cite[3.16]{hardyInequalities1964}.
\begin{lemma}[\cite{matkowskiConverseMinkowskisInequality1990}]\label{pro:generalized minkowski}
    Let $\varphi:\ffss\to\ffss$ be an increasing bijection. Fix an integer $J\gee2$, and fix $\lambdac=(\lambda_1,\dots,\lambda_J)\in\Delta^J$. Then $\vphi\in\MI$ if and only if
    \begin{align*}
        \vphi^{-1}\xkh{\sum_{j=1}^J\lambda_j\vphi(x_j+y_j)}\lee\vphi^{-1}\xkh{\sum_{j=1}^J\lambda_j\vphi(x_j)}+\vphi^{-1}\xkh{\sum_{j=1}^J\lambda_j\vphi(y_j)},\\
        \forall x_j,y_j\gee0,j\in\{1,\dots,J\}.
    \end{align*}
    Moreover, if the increasing bijection $\vphi\in C^2(\ffss)$, $\vphi''>0$ in $(0,+\infty)$ and the function $\vphi'/\vphi''$ is superadditive in $(0,+\infty)$, then $\vphi\in\MI$.
\end{lemma}

\begin{remark}\label{rm:MI examples}
    It is obvious that $\vphi(x)=x^p\in\MI$ with $p\gee1$ by the classical Minkowski inequality. Here are other examples in $\MI$:
    \begin{enumerate}
        \item For any fixed $t_0>0$, let $\phi:(0,+\infty)\to(0,+\infty]$ be a superadditive function, then
        \begin{equation*}
            \vphi(x)=\int_0^xe^{\int_{t_0}^t\frac{1}{\phi(u)}\d u}\d t\in\MI;
        \end{equation*}
        \item $\vphi(x)=\frac{x^2}{x+1}\in\MI$~\cite{matkowskiConverseMinkowskisInequality1990};
        \item $\vphi(x)=\begin{cases}
            xe^{-1/x},&x>0 \\
            0,&x=0
            \end{cases}\in\MI$~\cite{matkowskiConverseMinkowskisInequality1990};
        \item $a\vphi\in\MI$ if $\vphi\in\MI$ and $a>0$.
    \end{enumerate}
\end{remark}

Moreover, we can easily obtain the generalized integral form of~\cref{pro:generalized minkowski}.
\begin{lemma}\label{pro:generalized integral minkowski}
    Let $(\ml{X},\muc)$ be a probability measure space, and let $f_0,f_1:\ml{X}\to\ffss^2$ be nonnegative measurable functions.
    If $\vphi\in\MI$, then
    \begin{equation*}
        \vphi^{-1}\xkh{\int_{\ml{X}}\vphi(f_0(\xc)+f_1(\xc))\d\muc(\xc)}\lee\sum_{i\in\{0,1\}}\vphi^{-1}\xkh{\int_{\ml{X}}\vphi(f_i(\xc))\d\muc(\xc)}.
    \end{equation*}
\end{lemma}
\begin{proof}
    Without loss of generality, we assume that
    \begin{equation*}
        u_i=\int_{\ml{X}}\vphi(f_i(\xc))\d\muc(\xc)\in[0,+\infty),\quad\forall i\in\{0,1\}.
    \end{equation*}
    If $\vphi\in\MI$, note that $F_\vphi$ is concave in $\ffss^2$, so there exist $a_0,a_1,b\in\ssy$ such that
    \begin{equation*}
        \begin{cases}
            F_\vphi(v_0,v_1)\lee a_0v_0+a_1v_1+b,\quad\forall v_0,v_1\gee0; \\
            F_\vphi(u_0,u_1)=a_0u_0+a_1u_1+b.
        \end{cases}
    \end{equation*}
    It follows that
    \begin{align*}
        \int_{\ml{X}}F_\vphi\xkh{\vphi(f_0(\xc)),\vphi(f_1(\xc))}\d\muc(\xc)
        &\lee\int_{\ml{X}}\zkh{\sum_{i\in\{0,1\}}a_i\vphi(f_i(\xc))+b}\d\muc(\xc) \\
        &=a_0u_0+a_1u_1+b \\
        &=F_\vphi\xkh{\int_{\ml{X}}\vphi(f_0(\xc))\d\muc(\xc),\int_{\ml{X}}\vphi(f_1(\xc))\d\muc(\xc)}.
    \end{align*}
    Since $\vphi$ is an increasing bijection, then
    \begin{align*}
        \int_{\ml{X}}\vphi(f_0(\xc)+f_1(\xc))\d\muc(\xc)&\lee\vphi\xkh{\sum_{i\in\{0,1\}}\vphi^{-1}\xkh{\int_{\ml{X}}\vphi(f_i(\xc))\d\muc(\xc)}}\\
        \implies\vphi^{-1}\xkh{\int_{\ml{X}}\vphi(f_0(\xc)+f_1(\xc))\d\muc(\xc)}&\lee\sum_{i\in\{0,1\}}\vphi^{-1}\xkh{\int_{\ml{X}}\vphi(f_i(\xc))\d\muc(\xc)}.
    \end{align*}
    This concludes the proof.
\end{proof}

Then we can define a series of metrics on the proper space.
\begin{lemma}\label{pro:vphi wasserstein distance}
    Let $(\ml{X},d)$ be a Polish metric space, and let $\vphi\in\MI$.
    The Wasserstein space of function $\vphi$ is defined as
    \begin{equation*}
        \gl_{\vphi}^d(\ml{X})\dy\dkh{\muc\in\gl({\ml{X}}):\exists\yc\in\ml{X},\st\int_\ml{X}\vphi(d(\xc,\yc))\d \muc(\xc)<+\infty},
    \end{equation*}
    where $\yc\in\ml{X}$ is arbitrary. This space does not depend on the choice of the point $\yc$.
    For any two probability measures $\muc_0,\muc_1$ on $\gl_{\vphi}^d(\ml{X})$, the corresponding Wasserstein distance of function $\vphi$ between $\muc_0$ and $\muc_1$ is defined by the formula
    \begin{equation}\label{eq:Wasserstein distance of function}
        \begin{aligned}
            \Wp_{\vphi}^d(\muc_0,\muc_1)\dy&\vphi^{-1}\xkh{\inf_{\gammac\in\Pi(\muc_0,\muc_1)}\int_{\ml{X}\times\ml{X}}\vphi(d(\xc_0,\xc_1))\d\gammac(\xc_0,\xc_1)} \\
        =&\inf\dkh{\vphi^{-1}\xkh{\md{E}[\vphi(d(\Xc_0,\Xc_1))]}:\law(\Xc_i)=\muc_i,i\in\{0,1\}}.
        \end{aligned}
    \end{equation}
    Then $\Wp_{\vphi}^d$ defines a finite distance on $\gl_{\vphi}^d(\ml{X})$.
\end{lemma}
\begin{proof}
    We extend the proof of the metric properties associated with~\cite[Definition 6.1]{villaniOptimalTransport2009} to our setting. Let us point out that there exists at least one minimizer of~\eqref{eq:Wasserstein distance of function} by~\cref{thm:existence of infimum}, so $\Wp_{\vphi}^d$ is well-defined. Let $\gammac$ be a transport plan between two elements $\muc_0$ and $\muc_1$ in $\gl_{\vphi}^d(\ml{X})$. Then by~\cref{pro:generalized integral minkowski}, the inequality
    \begin{align*}
       &\vphi^{-1}\xkh{\int_{\ml{X}\times\ml{X}}\vphi(d(\xc_0,\xc_1))\d\gammac(\xc_0,\xc_1)} \\
       \lee&\vphi^{-1}\xkh{\int_{\ml{X}\times\ml{X}}\vphi(d(\xc_0,\yc)+d(\xc_1,\yc))\d\gammac(\xc_0,\xc_1)} \\
       \lee&\sum_{i\in\{0,1\}}\vphi^{-1}\xkh{\int_{\ml{X}}\vphi(d(\xc_i,\yc))\d\muc_i(\xc_i)}
    \end{align*}
    shows that $\Wp_{\vphi}^d(\muc_0,\muc_1)<+\infty$. Similarly, for any $\muc\in\gl(\ml{X}),\yc,\yc'\in\ml{X}$,
    \begin{align*}
        \vphi^{-1}\xkh{\int_{\ml{X}}\vphi(d(\xc,\yc'))\d\muc(\xc)}\lee&\vphi^{-1}\xkh{\int_{\ml{X}}\vphi(d(\xc,\yc)+d(\yc,\yc'))\d\muc(\xc)} \\
       \lee&\vphi^{-1}\xkh{\int_{\ml{X}}\vphi(d(\xc,\yc))\d\muc(\xc)}+d(\yc,\yc')
    \end{align*}
    shows that the definition of $\gl_{\vphi}^d(\ml{X})$ does not depend on the choice of the point $\yc$. Next, we will proof that $\Wp_{\vphi}^d$ satisfies the axioms of a distance.
    
    First, $\Wp_{\vphi}^d$ is clearly symmetric and nonnegative. Second, let $\muc_0,\muc_1,\muc_2\in\gl(\ml{X})$, and let $(\Xc_0,\Xc_1)$ be an optimal coupling of $(\muc_0,\muc_1)$ and $(\Xc_1,\Xc_2)$ an optimal coupling of $(\muc_1,\muc_2)$ (for the cost function $\vphi\circ d$). By the Gluing Lemma in~\cite[Chapter 1]{villaniOptimalTransport2009}, there exist random variables $\Xc_0',\Xc_1',\Xc_2'$ with $\law(\Xc_0',\Xc_1')=\law(\Xc_0,\Xc_1)$ and $\law(\Xc_1',\Xc_2')=\law(\Xc_1,\Xc_2)$. In particular, $(\Xc_0',\Xc_2')$ is a coupling of $(\muc_0,\muc_2)$, so
    \begin{align*}
        \Wp_{\vphi}^d(\muc_0,\muc_2)&\lee\vphi^{-1}\xkh{\md{E}[\vphi(d(\Xc_0',\Xc_2'))]} \\
        &\lee\vphi^{-1}\xkh{\md{E}[\vphi(d(\Xc_0',\Xc_1')+d(\Xc_1',\Xc_2'))]} \\
        &\lee \vphi^{-1}\xkh{\md{E}[\vphi(d(\Xc_0',\Xc_1'))]}+\vphi^{-1}\xkh{\md{E}[\vphi(d(\Xc_1',\Xc_2'))]} \\
        &=\Wp_{\vphi}^d(\muc_0,\muc_1)+\Wp_{\vphi}^d(\muc_1,\muc_2),
    \end{align*}
    where the inequality leading to the third line is due to~\cref{pro:generalized integral minkowski}, and the last equality follows from the fact that $(\Xc_0',\Xc_1')$ and $(\Xc_1',\Xc_2')$ are optimal couplings. So $\Wp_{\vphi}^d$ satisfies the triangle inequality.

    Finally, assume that $\Wp_{\vphi}^d(\muc_0,\muc_1)=0$, then there exists a transference plan which is entirely concentrated on the diagonal $(\xc_0=\xc_1)$ in $\ml{X}\times\ml{X}$. So $\muc_1=\Id_\#\muc_0=\muc_0$.
\end{proof}

Now, for any $\ml{A}\subseteq\gl_{\vphi}^d(\ml{X})$, $\Wp_{\vphi}^d$ satisfies~\cref{ass:metric} by~\cref{pro:vphi wasserstein distance}. Moreover, for any finite metric $W$ on $\ml{A}$, we want to extend it to a global metric on $\ml{M}(\ml{A})$ such that its restriction on $\ml{A}$ coincides with $W$. This motivates the following definition, which generalizes the $\MW$ distance introduced in~\cite{delonWassersteinTypeDistanceSpace2020}.

\begin{definition}[Mixture distance]\label{def:Mixture distance}
    Let $\ml{A}\subseteq\gl(\Omega)$ be a dictionary of atoms. 
    We assume that $\ml{M}\dy\ml{M}(\ml{A})$ satisfies~\cref{ass:Identifiability}, and $W$ satisfies~\cref{ass:metric}, and $\varphi\in\MI$. We define the function $W_{\ml{A},\varphi}:\ml{M}\times\ml{M}\to\ffss$ as follows: for all $\nuc_i=\sum_{j_i=1}^{J_i} \lambda_{i,j_i}\muc_{i,j_i}\in\ml{M},\muc_{i,j_i}\in\ml{A},j\in\{1,\dots,J_i\},i\in\{0,1\}$,
    \begin{equation}\label{eq:Mixture distance}
        W_{\ml{A},\varphi}(\nuc_0,\nuc_1)
            \dy\vphi^{-1}\xkh{\inf_{w\dy(w_{j_0,j_1})\in\Pi(\lambdac_0,\lambdac_1)}\sum_{j_0=1}^{J_0}\sum_{j_1=1}^{J_1}w_{j_0,j_1}\vphi(W(\muc_{0,j_0},\muc_{1,j_1}))},
    \end{equation}
    where 
    \begin{align*}
        \Pi(\lambdac_0,\lambdac_1)\dy\bigg\{&w\dy(w_{j_0,j_1})_{
        \substack{
        1\lee j_i\lee J_i\\
        i\in\{0,1\}}
        }\in\ffss^{J_0\times J_1}:\\
        &\sum_{j_1=1}^{J_1}w_{j_0,j_1}=\lambda_{0,j_0},\sum_{j_0=1}^{J_0}w_{j_0,j_1}=\lambda_{1,j_1},\forall 1\lee j_i\lee J_i,i\in\{0,1\}\bigg\}.
    \end{align*}
\end{definition}
\begin{remark}
    By~\cref{thm:existence of infimum}, there exists at least one minimizer of~\eqref{eq:Mixture distance}. Then~\cref{ass:Identifiability} on $\ml{M}$ guarantees that $W_{\ml{A},\varphi}$ is well-defined and $\eval{W_{\ml{A},\varphi}}_{\ml{A}}=W$.
\end{remark}

Now the following results can be easily proved.
\begin{theorem}\label{thm:mixture metric}
    Let $\ml{A}\subseteq\gl(\Omega)$ be a dictionary of atoms. 
    We assume that $\ml{M}(\ml{A})$ satisfies~\cref{ass:Identifiability}, and $W$ satisfies~\cref{ass:metric}, and $\varphi\in\MI$. Then $W_{\ml{A},\varphi}$ defines a metric on $\ml{M}(\ml{A})$ indeed.
\end{theorem}
\begin{proof}
    This theorem generalizes the results of~\cite{villaniOptimalTransport2009,peyreComputationalOptimalTransport2019}.
    In what follows, we employ the notation defined in~\cref{def:Mixture distance}.

    First, $W_{\ml{A},\varphi}$ is clearly symmetric and nonnegative. Second, $W_{\ml{A},\varphi}(\muc_0,\muc_1)=0$ if and only if $\muc_0=\muc_1$.
    
    Next we prove the triangle inequality. Without loss of generality, let $\lambdac_i=(\lambda_{i,1},\dots,\lambda_{i,J_i})\in\Delta^{J_i}_{+}$, and let $\nuc_i=\sum_{j_i=1}^{J_i} \lambda_{i,j_i}\muc_{i,j_i}\in\ml{M}(\ml{A})$ for each $i\in\{0,1,2\}$. We will show that 
    \begin{equation*}
        W_{\ml{A},\varphi}(\nuc_1,\nuc_2)\lee W_{\ml{A},\varphi}(\nuc_0,\nuc_1)+W_{\ml{A},\varphi}(\nuc_0,\nuc_2).
    \end{equation*}
    For each $i\in\{1,2\}$, let $w^{0,i}\in\ffss^{J_0\times J_{i}}$ be a minimizer of~\eqref{eq:Mixture distance} for $(\nuc_0,\nuc_{i})$. Let us define
    \begin{equation*}
        w^{1,2}_{j_1,j_2}\dy\sum_{j_0=1}^{J_0}\frac{w^{0,1}_{j_0,j_1}w^{0,2}_{j_0,j_2}}{\lambda_{0,j_0}},\quad \forall j_i\in\{1,\dots,J_i\},i\in\{1,2\}.
    \end{equation*}
    A simple calculation leads to show that $w^{1,2}\dy(w^{1,2}_{j_1,j_2})_{\substack{1\lee j_i\lee J_i,\\i\in\{0,1\}}}\in\Pi(\lambdac_1,\lambdac_2)$. Then
    \begin{align*}
        \vphi(W_{\ml{A},\varphi}(\nuc_1,\nuc_2))&\lee\sum_{j_1=1}^{J_1}\sum_{j_2=1}^{J_2}w^{1,2}_{j_1,j_2}\vphi(W(\muc_{1,j_1},\muc_{2,j_2})) \\
        &=\sum_{j_1=1}^{J_1}\sum_{j_2=1}^{J_2}\sum_{j_0=1}^{J_0}\frac{w^{0,1}_{j_0,j_1}w^{0,2}_{j_0,j_2}}{\lambda_{0,j_0}}\vphi(W(\muc_{1,j_1},\muc_{2,j_2})) \\
        &\lee\sum_{\substack{
        1\lee j_i\lee J_i\\
        i\in\{0,1,2\}}}\frac{w^{0,1}_{j_0,j_1}w^{0,2}_{j_0,j_2}}{\lambda_{0,j_0}}\vphi(W(\muc_{0,j_0},\muc_{1,j_1})+W(\muc_{0,j_0},\muc_{2,j_2})), 
    \end{align*}
    using the triangular inequality for $W$. Note that $\vphi\in\MI$, by~\cref{pro:generalized minkowski}, we obtain
    \begin{align*}
        W_{\ml{A},\varphi}(\nuc_1,\nuc_2)&\lee\sum_{k=1,2}\vphi^{-1}\xkh{\sum_{\substack{
        1\lee j_i\lee J_i\\
        i\in\{0,1,2\}}}\frac{w^{0,1}_{j_0,j_1}w^{0,2}_{j_0,j_2}}{\lambda_{0,j_0}}\vphi(W(\muc_{0,j_0},\muc_{k,j_k}))} \\
        &=\sum_{k=1,2}\vphi^{-1}\xkh{\sum_{\substack{
        1\lee j_i\lee J_i\\
        i=0,k}}w^{0,k}_{j_0,j_k}\vphi(W(\muc_{0,j_0},\muc_{k,j_k}))} \\
        &=W_{\ml{A},\varphi}(\nuc_0,\nuc_1)+W_{\ml{A},\varphi}(\nuc_0,\nuc_2).
    \end{align*}
    Now $W_{\ml{A},\varphi}$ defines a metric on $\ml{M}(\ml{A})$ indeed.
\end{proof}
\begin{remark}[\cite{dussonWassersteintypeMetricGeneric2026}]
    Moreover, if $\vphi(x)=x^p$ with $p\gee1$ and $(\ml{A},W)$ is a geodesic space, then $(\ml{M}(\ml{A}),W_{\ml{A},\varphi})$ is also a geodesic space.
\end{remark}

\section{Marginal problems}
A key challenge in computing optimal transport between mixture models is handling continuous marginal constraints. To resolve this, this section reduces continuous formulations to tractable discrete problems. Specifically,~\cref{thm:two marginal equivalent} and its multimarginal generalization (\cref{thm:multimarginal equivalent}) establish this equivalence for discrete couplings.

In view of~\cref{thm:mixture metric}, we focus on the specific case where $\ml{A}\subseteq\gl^d_\vphi(\Omega)$ for some $\vphi\in\MI$ and a distance function $d:\Omega\times\Omega\to\ffss$, and where the atomic metric $W$ is given by the Wasserstein metric $\Wp_\vphi^d$, i.e.
\begin{equation*}
    W(\muc_0,\muc_1)\dy\vphi^{-1}\xkh{\inf_{\gammac\in\Pi(\muc_0,\muc_1)}\int_{\Omega\times\Omega}\vphi(d(\xc_0,\xc_1))\d\gammac(\xc_0,\xc_1)},\quad\forall \muc_0,\muc_1\in\ml{A},
\end{equation*}
which is a well-defined finite metric on $\gl^{d}_\vphi(\Omega)$ by~\cref{pro:vphi wasserstein distance}. By~\cref{thm:existence of infimum}, there exists at least one solution to the above optimal transport problem, and we denote the set of minimizers by $\Gamma_{\vphi}^d(\muc_0,\muc_1)$. To begin with, we will show that the corresponding two-marginal problem can boil down to a simpler discrete optimal transport problem.
\subsection{Two-marginal problems}
It will be useful to introduce additional notation:
\begin{align*}
    \Gamma_{\vphi}^d(\ml{A})\dy&\bigcup_{\muc_0,\muc_1\in\ml{A}}\Gamma_{\vphi}^d(\muc_0,\muc_1), \\
    \Pi(\ml{A})\dy&\bigcup_{\muc_0,\muc_1\in\ml{A}}\Pi(\muc_0,\muc_1), \\
    \ml{M}\Gamma_{\vphi}^d(\ml{A})\dy&\ml{M}(\Gamma_{\vphi}^d(\ml{A}))
    \subseteq\ml{M}\Pi(\ml{A})\dy\ml{M}(\Pi(\ml{A}))\subseteq\gl(\Omega\times\Omega).
\end{align*}
\begin{theorem}\label{thm:two marginal equivalent}
    For each $i\in\{0,1\}$, let $\nuc_i=\sum_{j_i=1}^{J_i} \lambda_{i,j_i}\muc_{i,j_i}\in\ml{M}(\ml{A}),\lambdac_i=(\lambda_{i,1},\dots,\lambda_{i,J_i})\in\Delta^{J_i}$. Let us define
    \begin{equation}\label{eq:admissible mixture distance}
        \MWp_{\ml{A},\varphi}^d(\nuc_0,\nuc_1)\dy\vphi^{-1}\xkh{\inf_{\gammac\in\ml{M}\Pi(\ml{A})\cap\Pi(\nuc_0,\nuc_1)}\int_{\Omega\times\Omega}\vphi(d(\xc_0,\xc_1))\d\gammac(\xc_0,\xc_1)}.
    \end{equation}
    Then, under~\cref{ass:Identifiability},~\eqref{eq:admissible mixture distance} is equivalent to~\eqref{eq:Mixture distance}, i.e. $\MWp_{\ml{A},\varphi}^d(\nuc_0,\nuc_1)=W_{\ml{A},\varphi}(\nuc_0,\nuc_1)$ with $W\dy\Wp_\vphi^d$. Moreover, the sets of minimizers of~\eqref{eq:admissible mixture distance} is equal to the set of measures $\gammac$ which can be written as
    \begin{equation}\label{eq:mixture minimizer}
        \gammac=\sum_{j_0=1}^{J_0}\sum_{j_1=1}^{J_1}w_{j_0,j_1}^\ast\gammac_{j_0,j_1},
    \end{equation}
    with $\wc^\ast\dy(w_{j_0,j_1}^\ast)_{\substack{1\lee j_i\lee J_i\\
    i\in\{0,1\}}}$ a minimizer of~\eqref{eq:Mixture distance}, and $\gammac_{j_0,j_1}\in\Gamma_{\vphi}^d(\muc_{0,j_0},\muc_{1,j_1})$ for all $j_i\in\{1,\dots,J_i\},i\in\{0,1\}$.
\end{theorem}
\begin{proof}
    Our proof combines the approaches of~\cite{delonWassersteinTypeDistanceSpace2020} and~\cite{dussonWassersteintypeMetricGeneric2026}, extending the latter's result from the specific case of $\vphi(x)=x^p$ with $p\gee1$ to the collection of functions $\MI$.
 
    Without loss of generality, we assume that the measures $\muc_{0,j_0}$ (respectively $\muc_{1,j_1}$) are all distinct, and that $\lambdac_i\in\Delta^{J_i}_+$ for all $i\in\{0,1\}$.
    First, let $\wc^\ast\dy(w^\ast_{j_0,j_1})_{\substack{1\lee j_i\lee J_i\\
    i\in\{0,1\}}}$ be a solution to~\eqref{eq:Mixture distance}. For any $\gammac_{j_0,j_1}\in\Gamma_{\vphi}^d(\muc_{0,j_0},\muc_{1,j_1}),j_i\in\{1,\dots,J_i\},i\in\{0,1\}$, let us define $\gammac^\ast\dy\sum_{\substack{1\lee j_i\lee J_i\\
    i\in\{0,1\}}}w^\ast_{j_0,j_1}\gammac_{j_0,j_1}$. There holds $\gammac^\ast\in\ml{M}\Gamma_{\vphi}^d(\ml{A})\cap\Pi(\nuc_0,\nuc_1)$. Moreover,
    \begin{align*}
        \vphi(\MWp_{\ml{A},\varphi}^d(\nuc_0,\nuc_1))&\lee \sum_{j_0=1}^{J_0}\sum_{j_1=1}^{J_1}w^\ast_{j_0,j_1}\int_{\Omega\times\Omega}\vphi(d(\xc_0,\xc_1))\d\gammac_{j_0,j_1}(\xc_0,\xc_1) \\
        &=\sum_{j_0=1}^{J_0}\sum_{j_1=1}^{J_1}w^\ast_{j_0,j_1}\vphi(\Wp_{\vphi}^d(\muc_{0,j_0},\muc_{1,j_1})) \\
        &=\vphi(W_{\ml{A},\varphi}(\nuc_0,\nuc_1)).
    \end{align*}
    Second, for any $\gammac\in\ml{M}\Pi(\ml{A})\cap\Pi(\nuc_0,\nuc_1)$. Then, there exists $J\in\zzs$ such that $\gammac\in\sum_{j=1}^J\lambda_j\gammac_j$ with $\gammac_j\in\Pi(\ml{A})$. Using that ${\proj_0}_{\#}\gammac=\nu_0$ and ${\proj_1}_{\#}\gammac=\nu_1$ where $\proj_i$ stand for the projection map $(x_0,\dots,x_i,\dots)\mapsto x_i$ for any $i\in\zrs$, we obtain
    \begin{equation*}
        \sum_{j=1}^J\lambda_j{\proj_0}_{\#}\gammac_j=\sum_{j_0=1}^{J_0}\lambda_{0,j_0}\muc_{0,j_0}.
    \end{equation*}
    Using the fact that for all $j\in\{1,\dots,J\}$, ${\proj_0}_{\#}\gammac_j\in\ml{A}$ and using the identifiability~\cref{ass:Identifiability}, we obtain that for each $j\in\{1,\dots,J\}$, there exists $j_0'\in\{1,\dots,J_0\}$ such that ${\proj_0}_{\#}\gammac_j=\muc_{0,j_0'}$, and the set $\Lambda_{0,j_0}\dy\{j:{\proj_0}_{\#}\gammac_j=\muc_{0,j_0},1\lee j\lee J\}$ is a nonempty set for each $j_0\in\{1,\dots,J_0\}$. Similarly, for each $j\in\{1,\dots,J\}$, there exists $j_1'\in\{1,\dots,J_1\}$ such that ${\proj_1}_{\#}\gammac_j=\muc_{1,j_1'},\lambda_j=\lambda_{1,j_1'}$, and the set $\Lambda_{1,j_1}\dy\{j:{\proj_1}_{\#}\gammac_j=\muc_{1,j_1},1\lee j\lee J\}$ is a nonempty set for each $j_1\in\{1,\dots,J_1\}$. Now for each $j_i\in\{1,\dots,J_i\},i\in\{0,1\}$, we can define $w_{j_0,j_1}'\dy\sum_{j\in\Lambda_{0,j_0}\cap\Lambda_{1,j_1}}\lambda_j$, and
    \begin{equation*}
        \gammac_{j_0,j_1}'\dy\begin{cases}
            \frac{1}{w'_{j_0,j_1}}\sum_{j\in\Lambda_{0,j_0}\cap\Lambda_{1,j_1}}\gammac_j, & \mbox{if } w_{j_0,j_1}'>0;\\
            \gammac_{j_0,j_1}, & \mbox{if } w_{j_0,j_1}'=0.
        \end{cases}\in\Pi(\muc_{0,j_0},\muc_{1,j_1}).
    \end{equation*}
    Then $\gammac=\sum_{j_0=1}^{J_0}\sum_{j_1=1}^{J_1}w'_{j_0,j_1}\gammac'_{j_0,j_1}$ and $(w'_{j_0,j_1})_{\substack{1\lee j_i\lee J_i\\
    i\in\{0,1\}}}\in\Pi(\lambdac_0,\lambdac_1)$.
    Thus,
    \begin{align*}
        \int_{\Omega\times\Omega}\vphi(d(\xc_0.\xc_1))\d\gammac(\xc_0.\xc_1)&=\sum_{j_0=1}^{J_0}\sum_{j_1=1}^{J_1}w'_{j_0,j_1}\int_{\Omega\times\Omega}\vphi(d(\xc_0.\xc_1))\d\gammac'_{j_0,j_1}(\xc_0.\xc_1) \\
        &\gee\sum_{j_=1}^{J}w'_j \vphi(\Wp^d_{\varphi}(\muc_{0,j_0'},\muc_{1,j_1'})) \\
        &\gee\vphi(\Wp_{\ml{A},\varphi}(\nuc_0,\nuc_1)),
    \end{align*}
    which conclude $\MWp_{\ml{A},\varphi}^d(\nuc_0,\nuc_1)\gee\Wp_{\ml{A},\varphi}(\nuc_0,\nuc_1)$. 
    
    Finally, $\MWp_{\ml{A},\varphi}^d(\nuc_0,\nuc_1)=W_{\ml{A},\varphi}(\nuc_0,\nuc_1)$, and~\eqref{eq:mixture minimizer} is then a straightforward consequence of this proof.
\end{proof}

\subsection{Multimarginal problems}
When the atomic metric $W$ lacks a Euclidean structure, the mere issue of defining the barycenter of points becomes a challenging task. So for the general multimarginal problems, let us define the corresponding barycenter firstly.
\begin{lemma}[Borel barycenter map]\label{pro:Borel barycenter map}
    Let $\vphi\in\MI$ and $(\ml{X},d)$ be a locally compact Polish space. Then, given any integer $K\gee2$ and weights $\tauc\dy(\tau_k)_{0\lee k\lee K-1}\in\Delta^{K}$, there exists a Borel map $T_B^{\tauc}:\ml{X}^K\to\ml{X}$ that associate $(\xc_k)_{0\lee k\lee K-1}$ to a minimum of $\xc\mapsto\sum_{k=0}^{K-1}\tau_k\vphi(d(\xc,\xc_k))$. Such maps will be called {\rm Borel barycenter map}.

    Moreover, if for all $\xc\in\ml{X}$, $\yc\mapsto\vphi(d(\xc,\yc))$ is a strictly convex function with respect to $\yc$ on $\ml{X}$, then $T_B^{\tauc}$ is a continuous function on $\ml{Z}\dy\ml{X}^K$ and uniquely determined, where the product metric $d_\ml{Z}$ is defined by any norm $h$ on $\ffss^K$ such that
    \begin{equation*}
        d_\ml{Z}(\Xc,\Xc')\dy h(d(\xc_0,\xc_0'),\dots,d(\xc_{K-1},\xc_{K-1}')),
    \end{equation*}
    for all $\Xc\dy(\xc_0,\dots,\xc_{K-1}),\Xc'\dy(\xc_0',\dots,\xc_{K-1}')\in \ml{Z}$.
\end{lemma}
\begin{proof}
    The first part of this proof extends the argument of~\cite{legouicExistenceConsistencyWasserstein2017} to our generalized setting. For simplicity, we denote
    \begin{equation*}
        f_{\tauc}(\yc,\Xc)\dy\sum_{k=0}^{K-1}\tau_k\vphi(d(\yc,\xc_k)),\quad\forall\yc\in\ml{X},\Xc\in \ml{Z}.
    \end{equation*}
    Since $(\ml{X},d)$ is a locally compact Polish space, applying Theorem A.5 in~\cite{zimmerErgodicTheorySemisimple1984} with
    \begin{equation*}
        A\dy\dkh{(\xc_0,\dots,\xc_{K-1},\xc)\in \ml{Z}\times {\ml{X}}:f_{\tauc}(\xc,\Xc)\lee f_{\tauc}(\zc,\Xc),\forall\zc\in {\ml{X}}}.
    \end{equation*}
    shows the existence of a Borel section $f$ from $\pi_{\ml{Z}}(A)$ to $\ml{Z}\times {\ml{X}}$ of the projection $\pi_{\ml{Z}}:\ml{Z}\times {\ml{X}}\to \ml{Z}$. Then $T_B^{\tauc}\dy\pi_{\ml{X}}\circ f$ is a Borel barycenter map, where $\pi_{\ml{X}}:\ml{Z}\times {\ml{X}}\to {\ml{X}}$ denotes the projection.

    Secondly, if $\yc\mapsto\vphi(d(\xc,\yc))$ is a strictly convex function for all $\xc,\yc\in\ml{X}$, then $f_{\tauc}(\yc,\Xc)$ is a strictly convex continuous function with respect to $\yc$ on $\ml{X}$, so its global minimum point is unique. We therefore can define
    \begin{equation*}
        T_B^{\tauc}(\Xc)\dy\argmin_{\yc\in\ml{X}}f_{\tauc}(\yc,\Xc).
    \end{equation*}
    For fixed $\Xc\in \ml{Z},\varepsilon>0$, and any $\Xc'\dy(\xc_0',\dots,\xc_{K-1}')\in \ml{Z}$, we denote $\zc\coloneqq T_B^{\tauc}(\Xc),\zc'\coloneqq T_B^{\tauc}(\Xc')$. If $d(\zc,\zc')\gee\varepsilon$, since the strict convexity of $f_{\tauc}(\,\cdot\,,\Xc)$ and $\vphi\in\MI$, then there exists $\varepsilon_1>0$ such that $\vphi^{-1}(f_{\tauc}(\zc',\Xc))>\vphi^{-1}(f_{\tauc}(\zc,\Xc))+\varepsilon_1$. According to~\cref{pro:generalized minkowski}, for any $\yc\in {\ml{X}}$,
    \begin{equation*}
        \jdz{\vphi^{-1}(f_{\tauc}(\yc,\Xc))-\vphi^{-1}(f_{\tauc}(\yc,\Xc'))}\lee\vphi^{-1}\xkh{\sum_{k=0}^{K-1}\tau_k\vphi(d(\xc_k,\xc_k')))}.
    \end{equation*}
    It is straightforward to verify that $d_\ml{Z}$ defines a norm on $\ml{Z}$. By the equivalence of norms in finite dimensions, there exists $\varepsilon_2>0$ such that for any $d_\ml{Z}(\Xc,\Xc')<\varepsilon_2$,
    \begin{equation*}
        \max_{k\in\{0,\dots,K-1\}}{d(\xc_k,\xc_k')}<\vphi^{-1}\xkh{\frac{1}{K}\vphi(\frac{\varepsilon_1}{3})},
    \end{equation*}
    which implies
    \begin{equation*}
        \vphi^{-1}\xkh{\sum_{k=0}^{K-1}\tau_k\vphi(d(\xc_k,\xc_k')))}<\frac{\varepsilon_1}{3}.
    \end{equation*}
    Now if there exists $\Xc'\in \ml{Z}$ such that $d_\ml{Z}(\Xc,\Xc')<\varepsilon_2$ and $d(\zc,\zc')\gee\varepsilon$, then
    \begin{equation*}
        \frac{2\varepsilon_1}{3}>\vphi^{-1}(f_{\tauc}(\zc',\Xc))-\vphi^{-1}(f_{\tauc}(\zc,\Xc))+\vphi^{-1}(f_{\tauc}(\zc,\Xc'))-\vphi^{-1}(f_{\tauc}(\zc',\Xc'))>\varepsilon_1.
    \end{equation*}
    This is a contradiction, so $d_\ml{Z}(\Xc,\Xc')<\varepsilon_2$ implies $d(\zc,\zc')<\varepsilon$, therefore $T_B^{\tauc}$ is a continuous function and uniquely determined.
\end{proof}
\begin{theorem}[Barycenter and multimarginal problem]\label{thm:Barycenter and multimarginal problem}
    Let $(\ml{X},d)$ be a locally compact Polish space, and $\vphi\in\MI$. Given an integer $K\gee2$, $(\muc_k)_{0\lee k\lee K-1}\in\gl_\vphi^d(\ml{X})^K$ and weights $\tauc\dy(\tau_k)_{0\lee k\lee K-1}\in\Delta^{K}$, there exists a measure $\gammac_B\in\Gamma(\muc_0,\dots,\muc_{K-1})$ minimizing
    \begin{equation*}
        \gammac\mapsto\inf_{\xc\in\ml{X}}\int_{\ml{X}^K}\sum_{k=0}^{K-1}\tau_k\vphi(d(\xc,\xc_k))\d\gammac(\xc_0,\dots,\xc_{K-1}).
    \end{equation*}
    Moreover, let $T_B^{\tauc}:\ml{X}^K\to\ml{X}$ be a Borel barycenter map (as in~\cref{pro:Borel barycenter map}). Then, the measure $\muc_B\dy {T_B^{\tauc}}_{\#}\gammac_B$ is a barycenter of $(\muc_k)_{0\lee k\lee K-1}$. Furthermore, if this map is unique, any barycenter $\muc_B$ is of the form $\muc_B={T_B^{\tauc}}_{\#}\gammac_B$.
\end{theorem}
\begin{proof}
    This proof extends the argument of~\cite{aguehBarycentersWassersteinSpace2011,legouicExistenceConsistencyWasserstein2017} to our setting. Under the assumptions that $\varphi\in\MI$ and that $\ml{X}$ is locally compact, the existence of a solution to the multimarginal problem follows directly from~\cref{thm:existence of infimum}.

    Let $\gammac_B$ be a solution of the multimarginal problem, and set $\muc_B \coloneqq {T_B^{\tauc}}_{\#}\gammac_B$. For simplicity, we write $\Xc \coloneqq (\xc_0, \dots, \xc_{K-1})$. Then, by the definition of the Wasserstein distance associated with $\varphi$, for each $k\in\{0,\dots,K-1\}$,
    \begin{equation*}
        \vphi(\Wp_\vphi^d(\muc_k,\muc_B))\lee\int_{\ml{X}^K}\vphi(d(\xc_k,T_B^{\tauc}(\Xc)))\d\gammac_B(\Xc),
    \end{equation*}
    and consequently,
    \begin{equation}\label{eq:01}
        \sum_{k=0}^{K-1}\tau_k\vphi(\Wp_\vphi^d(\muc_k,\muc_B))\lee\int_{\ml{X}^K}\sum_{k=0}^{K-1}\tau_k\vphi(d(\xc_k,T_B^{\tauc}(\Xc)))\d\gammac_B(\Xc).
    \end{equation}

    Furthermore, for any $\nuc\in\gl_\varphi^d(\ml{X})^K$, let $\pi_k\in\Gamma(\muc_k,\nuc_k)$ be an optimal transport plan between $\muc_k$ and $\nuc$ for each $k\in\{0,\dots,K-1\}$.
    By the disintegration theorem, for any $k\in\{0,\dots,K-1\}$, there exists a family of conditional measures $\{\muc_k^{\yc}\}$ defined for $\nuc$-almost every $\yc$, which satisfies $\pi_k(\d \xc_k, \d \yc) = \muc_k^{\yc}(\d \xc_k) \nuc(\d \yc)$. We then set,
    \begin{equation*}
        \thetac(\d\Xc,\d\yc)\dy\muc_0^{\yc}(\d\xc_0)\times\dots\times\muc_{K-1}^{\yc}(\d\xc_{K-1})\times \nuc(\d\yc),
    \end{equation*}
    and let $\gammac'$ be the law of the $K$ first marginals of $\thetac$. Then, by the construction of $\thetac$,
    \begin{align}
        \sum_{k=0}^{K-1}\tau_k\vphi(\Wp_\vphi^d(\muc_k,\nuc))=&\sum_{k=0}^{K-1}\tau_k\int_{\ml{X}^{K+1}}\vphi(d(\xc_k,\yc)))\d\thetac(\Xc,\yc) \nonumber\\
        \gee&\sum_{k=0}^{K-1}\tau_k\inf_{\zc\in\ml{X}}\int_{\ml{X}^{K+1}}\vphi(d(\xc_k,\zc)))\d\thetac(\Xc,\yc) \nonumber\\
        =&\sum_{k=0}^{K-1}\tau_k\int_{\ml{X}^{K+1}}\vphi(d(\xc_k,T_B^{\tauc}(\Xc))))\d\thetac(\Xc,\yc) \label{eq:02}\\
        =&\sum_{k=0}^{K-1}\tau_k\int_{\ml{X}^{K}}\vphi(d(\xc_k,T_B^{\tauc}(\Xc))))\d\gammac'(\Xc) \nonumber\\
        \gee&\int_{\ml{X}^K}\sum_{k=0}^{K-1}\tau_k\vphi(d(\xc_k,T_B^{\tauc}(\Xc)))\d\gammac_B(\Xc) \nonumber\\
        \gee&\sum_{k=0}^{K-1}\tau_k\vphi(\Wp_\vphi^d(\muc_k,\muc_B)). \nonumber
    \end{align}
    where the last inequality follows from~\eqref{eq:01}. Since $\nuc$ is arbitrary, we conclude that ${T_B^{\tauc}}_{\#}\gammac_B$ is a barycenter.

    Moreover, by choosing $\nuc$ to be a barycenter,~\eqref{eq:02} becomes an equality, which implies that for $\thetac$-almost every $(\Xc, \yc)\in\ml{X}^K\times\ml{X}$,
    \begin{equation*}
        \sum_{k=0}^{K-1}\tau_k\vphi(d(\xc_k,\yc)))=\inf_{\zc\in\ml{X}}\sum_{k=0}^{K-1}\tau_k\vphi(d(\xc_k,\zc)))=\sum_{k=0}^{K-1}\tau_k\vphi(d(\xc_k,T_B^{\tauc}(\Xc)))).
    \end{equation*}
    Thus, if the barycenter map $T_B^{\tauc}$ is unique, then $T_B^{\tauc}(\Xc)=\yc$ for $\thetac$-almost everywhere, which implies that ${T_B^{\tauc}}_{\#}\gammac' = \nuc$. Furthermore, the optimality of $\nuc$, together with~\eqref{eq:01}, shows that $\nuc$ is a solution to the multimarginal problem.
\end{proof}

Now the theory naturally extends to multimarginal problems by the following assumption.
\begin{assumption}\label{ass:unique continuous barycenter}
    Let $(\ml{X},d)$ be a locally compact Polish space whose Borel barycenter map $T_B^{\tauc}$ is uniquely determined and continuous function for any integer $K\gee2$ and any weight vector $\tauc=(\tau_k)_{0\lee k\lee K-1}\in\Delta^{K}$.
\end{assumption}
\begin{definition}
    Let $(\ml{X},d)$ satisfy~\cref{ass:unique continuous barycenter}. For any integer $K\gee 2$, $\varphi\in\MI$, $\tauc\coloneqq(\tau_k)_{0\lee k \lee K-1}\in\Delta^{K}$, and $\muc_0,\dots,\muc_{K-1}\in\gl(\ml{X})$, we define
    \begin{equation}\label{eq:multimarginal problem}
        \Wp_{\vphi}^{d,\tauc}(\muc_0,\dots,\muc_{K-1})\dy\vphi^{-1}\xkh{\inf_{\gammac\in\Pi(\muc_0,\dots,\muc_{K-1})}c_\vphi^{d,\tauc}(\Xc)\d\gammac(\Xc)},
    \end{equation}
    where $\Xc=(\xc_0,\dots,\xc_{K-1})\in\ml{X}^{K}$ and
    \begin{equation}
        c_\vphi^{d,\tauc}(\Xc)\dy\sum_{k=0}^{K-1}\tau_k\vphi(d(\xc_k,T_B^{\tauc}(\xc_0,\dots,\xc_{K-1}))).
    \end{equation}
    Since $T_B$ is continuous, then, by~\cref{thm:existence of infimum}, there exists at least one solution to the above problem. The set of minimizers of~\eqref{eq:multimarginal problem} is denoted by $\Gamma_{\varphi}^{d,\tauc}(\muc_0,\dots,\muc_{K-1})$.
\end{definition}

Let $\ml{A}\subseteq\gl(\Omega)$ be a dictionary of atoms. Similarly, we introduce further notation:
\begin{align*}
    \Gamma_{\vphi,K}^{d,\tauc}(\ml{A})\dy&\bigcup_{\muc_0,\dots,\muc_{K-1}\in\ml{A}}\Gamma_{\vphi}^{d,\tauc}(\muc_0,\dots,\muc_{K-1}), \\
    \Pi_{K}(\ml{A})\dy&\bigcup_{\muc_0,\dots,\muc_{K-1}\in\ml{A}}\Pi(\muc_0,\dots,\muc_{K-1}), \\
    \ml{M}\Gamma_{\vphi,K}^{d,\tauc}(\ml{A})\dy&\ml{M}(\Gamma_{\vphi,K}^{d,\tauc}(\ml{A}))
    \subseteq\ml{M}\Pi_{K}(\ml{A})\dy\ml{M}(\Pi_{K}(\ml{A}))\subseteq\gl(\Omega^{K}).
\end{align*}

Following an approach similar to that in the previous subsection, we define the multimarginal transport problem.
\begin{definition}
    Let $(\ml{X},d)$ satisfy~\cref{ass:unique continuous barycenter}, and let $\ml{A}\subseteq\gl(\ml{X})$ be a dictionary of atoms. For any integer $K\gee2$, $\vphi\in\MI$, $\tauc\dy(\tau_k)_{0\lee\tau\lee K-1}\in\Delta^{K}$, and $\nuc_0,\dots,\nuc_{K-1}\in\ml{M}(\ml{A})$, we define
    \begin{equation}\label{eq:admissible multimarginal mixture}
        \MWp_{\ml{A},\vphi}^{d,\tauc}(\nuc_0,\dots,\nuc_{K-1})\dy\vphi^{-1}\xkh{\inf_{\gammac\in\ml{M}\Pi_{K}(\ml{A})\cap\Pi(\nuc_0,\dots,\nuc_{K-1})}c_\vphi^{d,\tauc}(\Xc)\d\gammac(\Xc)},
    \end{equation}
    where $\Xc=(\xc_0,\dots,\xc_{K-1})\in\ml{X}^{K}$.
\end{definition}

To state the corresponding result, we fix the integer $K\gee2$ and any integers $J_0,\dots,J_{K-1}\in\zzs$. We then define the multi-index set
\begin{equation*}
    \mgt{J} \coloneqq \{0, \dots, J_0\} \times \dots \times \{0, \dots, J_{K-1}\},
\end{equation*}
and denote its elements by $\jc = (j_0, \dots, j_{K-1}) \in \mgt{J}$.

\begin{theorem}\label{thm:multimarginal equivalent}
    Under the notation established above. Let $(\ml{X},d)$ satisfy~\cref{ass:unique continuous barycenter}, and let $\ml{A}\subseteq\gl(\ml{X})$ be a dictionary of atoms. For each $i\in\{0,\dots,K-1\}$ and $j_i\in\{1,\dots,J_i\}$, let $\muc_{i,j_i}\in\ml{A}$, $\nuc_i=\sum_{j_i=1}^{J_i} \lambda_{i,j_i}\muc_{i,j_i}\in\ml{M}(\ml{A})$, and $\lambdac_i=(\lambda_{i,1},\dots,\lambda_{i,J_i})\in\Delta^{J_i}$.
    Under~\cref{ass:Identifiability}, it holds that~\eqref{eq:admissible multimarginal mixture} is equivalent to the following discrete problem
    \begin{equation}\label{eq:multimarginal mixture distance}
        W_{\ml{A},\vphi}^{d,\tauc}(\nuc_0,\dots,\nuc_{K-1})\dy\vphi^{-1}\xkh{\inf_{\wc\in\Pi(\lambdac_0,\dots,\lambdac_{K-1})}\sum_{\jc\in\mgt{J}}w_{\jc}\vphi(\Wp_{\vphi}^{d,\tauc}(\muc_\jc))},
    \end{equation}
    where $w_\jc \coloneqq w_{j_0,\dots,j_{K-1}}$, $\muc_\jc\dy(\muc_{0,j_0},\dots,\muc_{K-1,j_{K-1}})$, and
    \begin{align*}
        \Pi(\lambdac_0,\dots,\lambdac_{K-1})\dy\bigg\{\wc=(w_{j_0,\dots,j_{K-1}})_{
        \substack{
        0\lee i\lee K-1\\
        1\lee j_i\lee J_i}
        }\in\ffss^{J_0\times\dots\times J_{K-1}}:
        \qquad\qquad\qquad\\
        \sum_{\substack{
        i'\neq i\\
        j_{i'}\in\{1,\dots,J_{i'}\}}}w_{j_0,\dots,j_{K-1}}=\lambda_{i,j_i},\forall i\in\{0,\dots,K-1\},j_{i}\in\{1,\dots,J_{i}\}\bigg\},
    \end{align*}
    i.e. $\MWp_{\ml{A},\vphi}^{d,\tauc}(\nuc_0,\dots,\nuc_{K-1})=W_{\ml{A},\vphi}^{d,\tauc}(\nuc_0,\dots,\nuc_{K-1})$. As a consequence, any minimizer of~\eqref{eq:admissible multimarginal mixture} can be written as
    \begin{equation}\label{eq:multimarginal mixture solution}
        \gammac^\ast\dy\sum_{\jc\in\mgt{J}}w^\ast_\jc\gammac^\ast_\jc,
    \end{equation}
    where $(w^\ast_\jc)_{\jc\in\mgt{J}}$ is a minimizer of~\eqref{eq:multimarginal mixture distance} and $\gammac^\ast_\jc\in\Gamma_{\vphi}^{d,\tauc}(\muc_\jc)$ for all $\jc\in\mgt{J}$. Moreover, any minimizer $(w^\ast_\jc)_{\jc\in\mgt{J}}$ to~\eqref{eq:admissible multimarginal mixture} contains at most $J_0+\dots+J_{K-1}-K+1$ nonzero components.
    Finally, any barycenter of $(\muc_0,\dots,\muc_{K-1})$ with barycentric weights $\tauc$ can be written as
    \begin{equation}\label{eq:barycenter mixture}
        \Bary_\tauc(\nuc_0,\dots,\nuc_{K-1})=\sum_{\jc\in\mgt{J}}w^\ast_\jc\Bary_\tauc(\muc_\jc),
    \end{equation}
    where $\Bary_\tauc(\muc_\jc)$ is the barycenter of $\muc_\jc$ with barycentric weights $\tauc$, i.e. $\Bary_\tauc(\muc_\jc)\dy {T_B^{\tauc}}_{\#}\gammac^\ast_\jc$ for all $\jc\in\mgt{J}$.
\end{theorem}
\begin{proof}
    By~\cref{thm:existence of infimum}, there exists at least one minimizer of~\eqref{eq:multimarginal mixture distance}. Then~\cref{ass:Identifiability} on $\ml{M}$ guarantees that $W_{\ml{A},\vphi}^{d,\tauc}$ is well-defined and $\eval{W_{\ml{A},\vphi}^{d,\tauc}}_{\ml{A}}=\Wp_{\vphi}^{d,\tauc}$. The proof of the equivalence between the discrete and continuous problems in~\cref{thm:two marginal equivalent} remains valid. Moreover, since the discrete problem is a linear program with $J_0+\dots+J_{K-1}-K+1$ affine constraints, there exists a solution with at most $J_0+\dots+J_{K-1}-K+1$ nonzero components~\cite[Theorem 2]{anderesDiscreteWassersteinBarycenters2016}, which lead to the minimizer of~\eqref{eq:admissible multimarginal mixture}. Finally, the structure of the barycenter follows directly from~\eqref{eq:multimarginal mixture solution},~\cref{ass:unique continuous barycenter} and~\cref{thm:Barycenter and multimarginal problem}.
\end{proof}

\section{Multivariate location-scale families}\label{sec:lsm}
To facilitate practical implementation, particularly within the subsequent domain adaptation framework, this section specializes our theoretical analysis to multivariate location-scale families. Our contributions are twofold:~\cref{thm:cost function} establishes the precise conditions under which affine transformations are optimal transport maps, while~\cref{thm:W_phi general} provides explicit analytical formulas for the optimal cost. Notably, for specific function families such as GMMs, these reduce to closed-form formulas.

Location-scale measures are generated from affine transformations of a given probability measure. By considering the multivariate version, we define the corresponding set of atoms as follows.
\begin{definition}[Affine location-scale atoms]
    We define the set of affine location-scale atoms generated from $\muc\in\gl_2(\ssyn)$ as 
    \begin{equation*}
        \ml{ALS}(\muc)=\dkh{T_{\#}\muc:T(\xc)=A\xc+\bc,\forall\xc\in\ssyn,A\in\ssynn,\bc\in\ssyn,\det(A)\neq0}.
    \end{equation*}
\end{definition}

We then rewrite a lemma on affine location-scale measures, which related to the $L^2$-Wasserstein distance.
\begin{lemma}[\cite{dowsonFrechetDistanceMultivariate1982}]\label{pro:affine location-scale L2}
    Let $\muc\in\gl_2(\ssyn)$ with mean $\mc\in\ssyn$ and covariance matrix $\Sigma\in\psd^n$. Let $\muc_0,\muc_1\in\ml{ALS}(\muc)$, and there exists $T_i:\ssyn\to\ssyn$ with $A_i\in\ssynn,\bc_i\in\ssyn$ such that
    \begin{equation*}
        T_i(\xc)=A_i\xc+\bc_i,\;\;\muc_i={T_i}_{\#}\muc,\quad\forall \xc\in\ssyn,i\in\{0,1\}.
    \end{equation*}
    Then, $\muc_i$ has mean $\mc_i$ and covariance $\Sigma_i$ defined by 
    \begin{equation*}
        \mc_i\dy A_i\mc+\bc_i,\Sigma_i\dy A_i\Sigma A_i^\zz,\quad\forall i\in\{0,1\}.
    \end{equation*}
    Moreover, the $L^2$-Wasserstein distance squared between $\muc_0$ and $\muc_1$ satisfies 
    \begin{equation*}
        \W(\muc_0,\muc_1)^2\gee \fs{\mc_0-\mc_1}_2^2+\Tr{\Sigma_0+\Sigma_1-2(\Sigma_0^{\frac{1}{2}}\Sigma_1\Sigma_0^{\frac{1}{2}})^{\frac{1}{2}}},
    \end{equation*}
    with equality if and only if the transport map
    \begin{equation*}
        T:\ssyn\to\ssyn,\quad
        T(\xc)=\Sigma_0^{-\frac{1}{2}}(\Sigma_0^{\frac{1}{2}}\Sigma_1\Sigma_0^{\frac{1}{2}})^{\frac{1}{2}}\Sigma_0^{-\frac{1}{2}}(\xc-\mc_0)+\mc_1,
    \end{equation*}
    is such that $T_{\#}\muc_0=\muc_1$.
\end{lemma}
\begin{remark}
    If $A_0,A_1\in\ffdjz^n\dy\{\diag(x_1,\dots,x_n):x_{j}\gee0,\forall 1\lee j\lee n\}$, or more generally, if $A_0^\zz A_1=A_1^\zz A_0\in\psd^n$, then the inequality in~\cref{pro:affine location-scale L2} is tight. On the other hand, if this inequality is tight for all $A_0,A_1\in\zdz^n$, then $\muc$ is a shift of a spherically invariant distribution almost surely~\cite{cuestaalbertosOptimalCouplingMultivariate1993}.
\end{remark}

Let us consider more general cases in which the affine transformation constitutes an optimal transport map.

\begin{theorem}[Characterization of Costs for Positive Definite Affine Optimal Maps]\label{thm:cost function}
Let $\muc\in\gl(\ssyn)$ be absolutely continuous with respect to the Lebesgue measure $\ml{L}^n$, whose density is strictly positive on a non-empty open set $U \subseteq \ssyn$. Let $\mht{P} \subseteq \zdz^n \times \ssyn$. For each $(A,\bc)\in\mht{P}$, we define $T_{A,\bc}(\xc)\dy A\xc+\bc$ and $\muc_{A,\bc}\dy {T_{A,\bc}}_{\#}\muc$.
Let the cost function $c: \ssyn \times \ssyn \to \md{R}$ be translation-invariant, i.e., $c(\xc,\yc) \dy \tilde{c}(\xc-\yc)$, where $\tilde{c} \in C(\ssyn)$ is twice continuously differentiable except on a finite set $\mht{S} \subset \ssyn$.
Assume that for each $(A,\bc)\in\mht{P}$, there exist upper semicontinuous functions $u_{A,\bc}\in L^1(\muc)$ and $v_{A,\bc}\in L^1(\muc_{A,\bc})$ such that $c(\xc,\yc)\gee u_{A,\bc}(\xc)+v_{A,\bc}(\yc)$ for all $\xc,\yc\in\ssyn$, and that
\begin{equation*}
    C_c(\muc,\muc_{A,\bc})\dy\inf_{\gammac\in\Pi(\muc,\muc_{A,\bc})}\int_{\ssyn\times\ssyn} c(\xc,\yc)\d \gammac(\xc,\yc)<+\infty.
\end{equation*}

\begin{itemize}
    \item[\textbf{(i)}] \textbf{Necessity}: Suppose that for any $(A, \bc) \in \mht{P}$, $T_{A,\bc}$ is an optimal transport map with respect to $c$. Assume that there exists at least one matrix $A$ with mutually distinct eigenvalues for $n \geqslant 2$, and $A \neq [1]$ for $n = 1$, such that $\{A\}\times\ssyn\subseteq\mht{P}$. 
    Then, $\tilde{c}$ is a globally convex function that is separable with respect to the orthonormal basis of eigenvectors of $A$. Specifically, if $A = Q D Q^\zz$ is the spectral decomposition, there exists a convex function $c_i$ for each $i\in\{1,\dots,n\}$ such that $\tilde{c}(\zc) = \sum_{i=1}^n c_i((Q^\zz \zc)_i)$ for all $\zc \in \ssyn$, where $(\,\cdot\,)_i$ denotes the $i$-th component of a vector. Furthermore, if $n \geqslant 2$, each $c_i$ is globally twice continuously differentiable, i.e., $c_i \in C^2(\md{R})$.
    
    \item[\textbf{(ii)}] \textbf{Sufficiency}: Conversely, if $\tilde{c}(\zc) = \sum_{i=1}^n c_i((Q^\zz \zc)_i)$ for all $\zc \in \ssyn$, where each $c_i \in C^2(\md{R})$ is convex, then for any SPD matrix $A$ diagonalized by $Q$ and any $\bc \in \ssyn$, the affine map $T_{A,\bc}(\xc) = A\xc + \bc$ is an optimal transport map with respect to $c$ between $\muc$ and $(T_{A,\bc})_\# \muc$. If, in addition, each $c_i$ is strictly convex, then $T_{A,\bc}$ is the unique optimal transport map.
    
    \item[\textbf{(iii)}] \textbf{Cross-Coupling and Quadratic Costs}: Additionally, suppose the hypotheses of (i) hold. If $\mht{P}$ contains another matrix $\bar{A}$ such that $\{\bar{A}\}\times\ssyn\subseteq\mht{P}$, let $\bar{A}_Q \dy Q^\zz \bar{A} Q$ be its representation in the $Q$-basis. If the $(i,j)$-th entry of $\bar{A}_Q$ is non-zero for some $i \neq j$, then $c_i$ and $c_j$ must be quadratic functions sharing the same second derivative, i.e., there exists a constant $\kappa \gee 0$ such that $c_i''(t) = c_j''(t) = \kappa$ for all $t \in \md{R}$.
\end{itemize}
\end{theorem}

\begin{proof}
The proof is divided into the following steps.

\textbf{Step 1: First-Order Condition and Regularity Bootstrapping.} \\
Fix an arbitrary $\bc \in \ssyn$ and let $T(\xc) = A\xc + \bc$. By Kantorovich duality~\cite[Theorem 5.10]{villaniOptimalTransport2009}, there exists a $c$-convex function $\psi$ such that the global inequality $\psi^c(\yc) - \psi(\xc) \lee c(\xc,\yc)$ holds for all $\xc,\yc \in \ssyn$, and the duality relation $\psi^c(T(\xc)) - \psi(\xc) = c(\xc, T(\xc))$ holds for $\muc$-a.e. $\xc \in \ssyn$.

Let $X_{\mht{S}} \dy \{\xc \in U : (\dwz-A)\xc - \bc \in \mht{S}\}$. Since $\mht{S}$ is finite and the kernel of $\dwz-A$ has dimension at most 1 (because $A$ has distinct eigenvalues for $n \geqslant 2$, and $A \neq [1]$ for $n = 1$), $X_{\mht{S}}$ is a finite union of affine subspaces of dimension at most 1. For $n \geqslant 2$, this implies the Lebesgue measure $\ml{L}^n(X_{\mht{S}}) = 0$. For $n = 1$, $A \neq [1]$ implies $\dwz-A$ is invertible, so $X_{\mht{S}}$ is finite and $\ml{L}^1(X_{\mht{S}}) = 0$. Thus, $U' \dy U \setminus X_{\mht{S}}$ is an open set of full Lebesgue measure in $U$.

We now establish the regularity of $\psi$. Since $\tilde{c} \in C(\ssyn)$ is twice continuously differentiable outside the finite set $\mht{S}$, it is locally Lipschitz continuous on $\ssyn$. Because the optimal map $T(\xc)$ is continuous, it is locally bounded on $U$. Since the supremum in the $c$-transform is attained at $T(\xc)$, this local boundedness, combined with the local Lipschitz continuity of $\tilde{c}$, guarantees that the $c$-convex potential $\psi$ is locally Lipschitz on $U$. By Rademacher's theorem, $\psi$ is differentiable $\ml{L}^n$-a.e. 

Let $\Omega_{\bc} \subseteq U'$ be the set of full measure where $\psi$ is differentiable and the duality relation holds. For any fixed $\xc \in \Omega_{\bc}$, define the auxiliary function $f_{\xc}(\xc') \dy c(\xc', T(\xc)) + \psi(\xc')$. The global inequality implies $f_{\xc}(\xc') \gee \psi^c(T(\xc))$ for all $\xc'$, while the duality relation ensures $f_{\xc}(\xc) = \psi^c(T(\xc))$. Thus, $f_{\xc}$ attains a global minimum at $\xc' = \xc$.

Since $\xc \in U'$ (an open set), it is an interior minimum. Furthermore, because $\xc \notin X_{\mht{S}}$, the point $\xc - T(\xc) \notin \mht{S}$, meaning $c(\,\cdot\,, T(\xc))$ is continuously differentiable at $\xc$. Since $\psi$ is also differentiable at $\xc$, Fermat's theorem for interior extrema dictates that $\nabla f_{\xc}(\xc) = \bm{0}$. This yields
\begin{equation}\label{eq:first_order}
    \nabla \psi(\xc) = -\nabla_{\xc} c(\xc, T(\xc)), \quad \forall \xc \in \Omega_{\bc}.
\end{equation}
Notice that the right-hand side is a $C^1$ vector field on $U'$ because $\tilde{c} \in C^2(\ssyn \setminus \mht{S})$. Since $\psi$ is locally Lipschitz, it belongs to the Sobolev space $W^{1,\infty}_{\text{loc}}(U')$, and its weak gradient $\nabla \psi$ exists almost everywhere.
Since $\nabla \psi$ coincides almost everywhere with the $C^1$ vector field $-\nabla_{\xc} c(\xc, T(\xc))$ on $U'$, it follows immediately 
that $\psi$ is classically continuously differentiable on $U'$, and 
$\nabla \psi(\xc) = -\nabla_{\xc} c(\xc, T(\xc))$ holds everywhere on $U'$.
The $C^1$-regularity of the right-hand side then guarantees that $\psi \in C^2(U')$.

\textbf{Step 2: Second-Order Condition and Globalization.} \\
Since $\psi \in C^2(U')$, we can differentiate~\eqref{eq:first_order} everywhere on $U'$ to obtain the Jacobian matrix relation: $\nabla^2 \psi(\xc) = -\nabla^2_{\xc\xc} c(\xc, T(\xc)) - \nabla^2_{\xc\yc} c(\xc, T(\xc)) A$. Concurrently, since $f_{\xc} \in C^2(U')$ attains a local minimum at $\xc$, its second-order necessary condition requires the pointwise Hessian to be PSD: $\nabla^2 f_{\xc}(\xc) = \nabla^2_{\xc\xc} c(\xc, T(\xc)) + \nabla^2 \psi(\xc) \succcurlyeq 0$. Substituting the former into the latter cancels $\nabla^2_{\xc\xc} c$, yielding $-\nabla^2_{\xc\yc} c(\xc, T(\xc)) A \succcurlyeq 0$. Using $c(\xc,\yc) = \tilde{c}(\xc-\yc)$, this becomes $\nabla^2 \tilde{c}(\xc-T(\xc)) A \succcurlyeq 0$. The symmetry of $\nabla^2 \psi(\xc)$ further implies that $\nabla^2 \tilde{c}(\xc-T(\xc))$ commutes with $A$.

To globalize this, let $\chi(\zc)$ be the indicator function of the set of points $\zc \in \ssyn$ where $\nabla^2 \tilde{c}(\zc)$ either does not exist, does not commute with $A$, or $\nabla^2 \tilde{c}(\zc) A$ is not PSD. From the local conditions, for every $\bc \in \ssyn$, we have $\chi((\dwz-A)\xc - \bc) = 0$ for $\ml{L}^n$-a.e. $\xc \in U$. Thus, $\int_U \chi((\dwz-A)\xc - \bc) \d\xc = 0$. 

Integrating this over all $\bc \in \ssyn$ and applying Tonelli's theorem to swap the order of integration, we obtain:
\begin{equation*}
    0 = \int_{\ssyn} \left( \int_U \chi((\dwz-A)\xc - \bc) \d\xc \right) \d\bc = \int_U \left( \int_{\ssyn} \chi((\dwz-A)\xc - \bc) \d\bc \right) \d\xc.
\end{equation*}
For any fixed $\xc \in U$, we perform the change of variables $\zc = (\dwz-A)\xc - \bc$. The Jacobian determinant of the map $\bc \mapsto \zc$ is $(-1)^n$, so $\d\bc = \d\zc$. The inner integral becomes $\int_{\ssyn} \chi(\zc) \d\zc$, which is entirely independent of $\xc$. Substituting this back yields:
\begin{equation*}
    0 = \int_U \left( \int_{\ssyn} \chi(\zc) \d\zc \right) \d\xc = \ml{L}^n(U) \int_{\ssyn} \chi(\zc) \d\zc.
\end{equation*}
Since $U$ is a non-empty open set, $\ml{L}^n(U) > 0$. Therefore, $\int_{\ssyn} \chi(\zc) \d\zc = 0$, meaning the commutativity and PSD properties hold for $\ml{L}^n$-a.e. $\zc \in \ssyn$. Because $\nabla^2 \tilde{c}$ is continuous on $\ssyn \setminus \mht{S}$, these properties must hold everywhere on the open set $\ssyn \setminus \mht{S}$.

\textbf{Step 3: Distributional Separability and Global Convexity.} \\
Since $\nabla^2 \tilde{c}(\zc)$ commutes with $A$ on $\ssyn \setminus \mht{S}$, and $A = Q D Q^\zz$ has distinct eigenvalues, $\nabla^2 \tilde{c}(\zc)$ is simultaneously diagonalized by $Q$. Let $\tc = (t_1, \dots, t_n)^\zz \in \ssyn$ and define $\hat{c}(\tc) \dy \tilde{c}(Q\tc)$. By the chain rule, $\nabla^2 \hat{c}(\tc) = Q^\zz \nabla^2 \tilde{c}(Q\tc) Q$ is a diagonal matrix for all $\tc \in \ssyn \setminus Q^\zz \mht{S}$.

This implies that all mixed classical second derivatives vanish: $\partial^2_{t_i t_j} \hat{c}(\tc) = 0$ for all $i \neq j$ on $\ssyn \setminus Q^\zz \mht{S}$. Note that any distribution supported on a finite set $Q^\zz \mht{S}$ must be a finite linear combination of Dirac $\delta$ distributions and their derivatives. However, the mixed partial derivative $\partial^2_{t_i t_j}$ of a continuous function cannot contain such singular distributions. Indeed, if $\partial^2_{t_i t_j} \hat{c}$ were to contain a Dirac $\delta$ mass $\delta_{\tc_0}$ (or its derivatives) at some point $\tc_0 \in Q^\zz \mht{S}$, integrating this relation would directly contradict the continuity of $\hat{c}$ around $\tc_0$. It follows that the distributional mixed derivatives $\partial^2_{t_i t_j} \hat{c}$ must vanish globally on $\ssyn$. By Schwartz's theorem on distributions, $\hat{c}$ is globally separable: $\hat{c}(\tc) = \sum_{i=1}^n c_i(t_i)$. The PSD condition $\nabla^2 \hat{c}(\tc) D \succcurlyeq 0$ outside a finite set ensures $c_i''(t_i) \gee 0$, proving global convexity.

For $n \geqslant 2$, we can further show that $c_i \in C^2(\md{R})$ for each $i\in\{1,\dots,n\}$. Fix an arbitrary $t_i \in \md{R}$. Since $\mht{S}$ is a finite set, $Q^\zz \mht{S} \subset \ssyn$ is also finite. Thus, the slice
\begin{equation*}
    \Xi_{t_i} \dy \dkh{ (t_1, \dots, t_{i-1}, t_{i+1}, \dots, t_n)^\zz \in \md{R}^{n-1} : (t_1, \dots, t_n)^\zz \in Q^\zz \mht{S} }
\end{equation*}
is finite. Since $n \geqslant 2$, the space $\md{R}^{n-1}$ is uncountable. We can therefore choose a point $(t_1^*, \dots, t_{i-1}^*, t_{i+1}^*, \dots, t_n^*)^\zz \in \md{R}^{n-1} \setminus \Xi_{t_i}$. By construction, the point $\tc^* \dy (t_1^*, \dots, t_{i-1}^*, t_i, t_{i+1}^*, \dots, t_n^*)^\zz$ does not belong to $Q^\zz \mht{S}$. Since $\ssyn \setminus Q^\zz \mht{S}$ is open, there exists an open neighborhood $V$ of $\tc^*$ such that $V \cap Q^\zz \mht{S} = \emptyset$. On this neighborhood $V$, the function $\hat{c}$ is of class $C^2$. Since $\hat{c}(\tc) = \sum_{k=1}^n c_k(t_k)$, we have
\begin{equation*}
    c_i(x) = \hat{c}(t_1^*, \dots, t_{i-1}^*, x, t_{i+1}^*, \dots, t_n^*) - \sum_{\ell \neq i} c_\ell(t_\ell^*)
\end{equation*}
for all $x$ in a neighborhood of $t_i$. Since the first term on the right-hand side is $C^2$ with respect to $x$ in a neighborhood of $t_i$, and the second term is a constant, it follows that $c_i$ is twice continuously differentiable in a neighborhood of $t_i$. Since $t_i \in \md{R}$ was chosen arbitrarily, we conclude that $c_i \in C^2(\md{R})$ for all $i \in \{1, \dots, n\}$ when $n \geqslant 2$.

\textbf{Step 4: Sufficiency and Uniqueness of the Optimal Map.} \\
Assume $\tilde{c}(\zc) = \sum_{i=1}^n c_i((Q^\zz \zc)_i)$ for all $\zc \in \ssyn$, where each $c_i \in C^2(\md{R})$ is convex. Let $A = Q D Q^\zz$ be any SPD matrix and $\bc \in \ssyn$. We construct the Kantorovich potential $\psi$ by defining its gradient: $\nabla \psi(\xc) \dy -\nabla_{\xc} c(\xc, T_{A,\bc}(\xc)) = -\nabla \tilde{c}((\dwz-A)\xc - \bc)$. The Jacobian of this vector field is $-\nabla^2 \tilde{c}((\dwz-A)\xc - \bc) (\dwz-A)$. Since $\nabla^2 \tilde{c}$ and $\dwz-A$ are both diagonalized by $Q$, they commute, making the Jacobian symmetric. By Poincaré's lemma, a global potential $\psi \in C^2(\ssyn)$ exists.

To prove that $T_{A,\bc}$ is an optimal transport map, it suffices to show that for any fixed $\xc \in \ssyn$, the objective function $f_{\xc}(\xc')$\,---\,redefined here as $c(\xc', T_{A,\bc}(\xc)) + \psi(\xc')$\,---\,attains its global minimum at $\xc' = \xc$. Without loss of generality, we can rotate the coordinate system by $Q^\zz$ so that $A = \diag(a_1, \dots, a_n)$ is diagonal with $a_i > 0$ for all $i\in\{1,\dots,n\}$, and $\tilde{c}(\zc) = \sum_{i=1}^n c_i(z_i)$. In this rotated coordinate system, the potential $\psi(\xc')$ decomposes as $\psi(\xc') = \sum_{i=1}^n \psi_i(x_i')$, where each $\psi_i$ satisfies $\psi_i'(x_i') = -c_i'((1-a_i)x_i' - b_i)$. Consequently, $f_{\xc}(\xc') = \sum_{i=1}^n f_i(x_i')$, where
\begin{equation*}
    f_i(x_i') \dy c_i(x_i' - a_i x_i - b_i) + \psi_i(x_i').
\end{equation*}
For each fixed $i \in \{1, \dots, n\}$, differentiating $f_i(x_i')$ with respect to $x_i'$ yields
\begin{equation*}
    f_i'(x_i') = c_i'(x_i' - a_i x_i - b_i) - c_i'((1-a_i)x_i' - b_i).
\end{equation*}
Let $q_1 \dy x_i' - a_i x_i - b_i$ and $q_2 \dy (1-a_i)x_i' - b_i$. Their difference is given by $q_1 - q_2 = a_i(x_i' - x_i)$. Since $c_i$ is convex, $c_i' \in C^1(\md{R})$ is non-decreasing. We analyze the sign of $f_i'(x_i')$ in two cases:
\begin{itemize}
    \item If $x_i' > x_i$, then $q_1 > q_2$, which implies $c_i'(q_1) \gee c_i'(q_2)$, and thus $f_i'(x_i') \gee 0$.
    \item If $x_i' < x_i$, then $q_1 < q_2$, which implies $c_i'(q_1) \lee c_i'(q_2)$, and thus $f_i'(x_i') \lee 0$.
\end{itemize}
This shows that $x_i$ is a global minimum of $f_i(x_i')$. Since this holds for each $i \in \{1, \dots, n\}$, the point $\xc' = \xc$ is a global minimum of $f_{\xc}(\xc')$, which proves that $T_{A,\bc}$ is an optimal transport map.

Now, assume in addition that each $c_i$ is strictly convex. Then $c_i'$ is strictly increasing. In this case, the inequalities above become strict:
\begin{itemize}
    \item If $x_i' > x_i$, then $q_1 > q_2$, which implies $c_i'(q_1) > c_i'(q_2)$, and thus $f_i'(x_i') > 0$.
    \item If $x_i' < x_i$, then $q_1 < q_2$, which implies $c_i'(q_1) < c_i'(q_2)$, and thus $f_i'(x_i') < 0$.
\end{itemize}
This establishes that $x_i$ is the unique global minimum of $f_i(x_i')$, and consequently, $\xc' = \xc$ is the unique global minimum of $f_{\xc}(\xc')$. This implies that the $c$-subdifferential of $\psi$ at $T_{A,\bc}(\xc)$ is the singleton $\{\xc\}$ for $\muc$-a.e. $\xc \in \ssyn$. By~\cite[Theorem 5.10]{villaniOptimalTransport2009}, the optimal transport map is unique.

\textbf{Step 5: Cross-Coupling Implies Quadratic Costs.} \\
Suppose $\mht{P}$ contains another matrix $\bar{A}$ such that $\{\bar{A}\}\times\ssyn\subseteq\mht{P}$. By the exact same globalization argument in Step 2, the Hessian $\nabla^2 \tilde{c}(\zc)$ must commute with $\bar{A}$ for $\ml{L}^n$-a.e. $\zc \in \ssyn$. Let $\tc = Q^\zz \zc$ and $\hat{c}(\tc) \dy \tilde{c}(Q\tc) = \sum_{k=1}^n c_k(t_k)$. The Hessian in the $Q$-basis is the diagonal matrix $\nabla^2 \hat{c}(\tc) = \diag(c_1''(t_1), \dots, c_n''(t_n))$. 

The commutativity condition $\nabla^2 \tilde{c}(\zc) \bar{A} = \bar{A} \nabla^2 \tilde{c}(\zc)$ transforms into $\nabla^2 \hat{c}(\tc) \bar{A}_Q = \bar{A}_Q \nabla^2 \hat{c}(\tc)$ for a.e. $\tc \in \ssyn$, where $\bar{A}_Q \dy Q^\zz \bar{A} Q$. Comparing the $(i,j)$-th entry of both sides yields
\begin{equation*}
    c_i''(t_i) (\bar{A}_Q)_{ij} = (\bar{A}_Q)_{ij} c_j''(t_j), \quad  \ml{L}^n\text{-a.e. } \tc \in \ssyn,
\end{equation*}
where $(\,\cdot\,)_{ij}$ denotes the entry in the $i$-th row and $j$-th column of a matrix. If $(\bar{A}_Q)_{ij} \neq 0$ for some $i \neq j$, we can divide by it to obtain $c_i''(t_i) = c_j''(t_j)$ for a.e. $t_i, t_j \in \md{R}$. Since the left-hand side depends only on $t_i$ and the right-hand side only on $t_j$, both must equal a common constant $\kappa$. Because $n \gee 2$ (since $i \neq j$), we know from Step 3 that $c_i, c_j \in C^2(\md{R})$, so this equality holds everywhere. The convexity of $c_i$ ensures $\kappa \gee 0$. Consequently, $c_i$ and $c_j$ are quadratic functions sharing the same second derivative $\kappa$. Note that any linear or constant terms in $c_i$ and $c_j$ do not affect the optimality of the transport map, as they merely add a constant to the total transport cost. Therefore, no further relationship between the lower-order terms of $c_i$ and $c_j$ can be deduced.
\end{proof}

\begin{remark}[Regularity Improvement for $n \geqslant 2$]
When $n \geqslant 2$, the conclusion that $c_i \in C^2(\md{R})$ for all $i\in\{1,\dots,n\}$ implies that the cost function $\tilde{c}$ is actually twice continuously differentiable on the entire space $\ssyn$, meaning the potential singular set $\mht{S}$ must be empty. This reveals that the existence of affine optimal transport maps for a full-space set of translations is so restrictive that it smooths out any isolated singularities of the cost function.
\end{remark}

\begin{remark}[Weaker Conditions on the Translation Set]
The assumption that $B_A = \ssyn$ can be relaxed to $B_A$ being merely dense in $\ssyn$, provided we additionally assume that $1$ is not an eigenvalue of $A$ for $n\gee2$. Under this condition, the matrix $\dwz-A$ is invertible, which maps the full-measure set $\Omega_{\bc}$ to a full-measure set $V_{\bc}$ in the domain $(\dwz-A)U' - \bc$. The countable union of $V_{\bc}$ over a dense set of translations $\bc$ is sufficient to cover $\ssyn$ almost everywhere.
\end{remark}

\begin{remark}[Rigidity to the Squared Euclidean Distance]
Building upon~\cref{thm:cost function}(iii), if the parameter set $\mht{P}$ is sufficiently rich such that the non-zero off-diagonal entries of the matrices in $\mht{P}$ (when represented in the $Q$-basis) form a connected graph over the indices $\{1, \dots, n\}$, then all components $c_i$ must share the same second derivative $\kappa \gee 0$. In this case, the cost function is forced to be exactly the squared Euclidean distance, $\tilde{c}(\zc) = \frac{\kappa}{2}\|\zc\|^2 + \bc_1^\zz \zc + b_2$, up to a linear and a constant term.
\end{remark}

\begin{remark}[Connection to Location-Scale Families]
In particular, $T_{A,\bc}$ is an optimal transport map with respect to the squared Euclidean distance if and only if $A \in \zdz^n$~\cite{svetlozart.MassTransportationProblems1998}. This aligns with our result when $c_i(t) = t^2$ for all $i\in\{1,\dots,n\}$. For a general cost function, for simplicity, one may choose $u_{A,\bc} \equiv v_{A,\bc} \equiv 0$, so that the condition $c(\xc,\yc) \gee u_{A,\bc}(\xc) + v_{A,\bc}(\yc)$ reduces to $\tilde{c} \gee 0$. Furthermore, by~\cref{thm:cost function}, we can define specific location-scale families generated from any suitable $\muc \in \gl(\ssyn)$ as follows:
\begin{align*}
    \ml{LS}(\muc) &\dy \dkh{T_\#\muc : T(\xc) = A\xc + \bc, \; \xc \in \ssyn, \; A \in \zdz^n, \; \bc \in \ssyn}, \\
    \ml{DLS}(\muc) &\dy \dkh{T_\#\muc : T(\xc) = A\xc + \bc, \; \xc \in \ssyn, \; A \in \zdjz^n, \; \bc \in \ssyn}.
\end{align*}
\end{remark}

A clear example of location-scale atoms is provided by a generative atom $\muc$, which is an absolutely continuous, radially symmetric probability measure on $\ssyn$ with finite second-order moments. For simplicity, we define its density as follows:
\begin{equation*}
    p_{g}(\xc)\dy g(\xc^\zz\xc),\quad\forall\xc\in\ssyn,
\end{equation*}
where $g$ is a nonnegative measurable function satisfying
\begin{equation*}
    \int_\ssyn g(\xc^\zz\xc)\d\xc=\frac{\pi^{n/2}}{\Gamma(\frac{n}{2})}\int_0^{+\infty} t^{\frac{n}{2}-1}g(t)\d t=1.
\end{equation*}
We consider the corresponding family $\ml{LS}(\muc)$ or $\ml{DLS}(\muc)$. 
Related to~\cref{pro:vphi wasserstein distance}, assume that there exist nonnegative, twice continuously differentiable (except at finitely many points) and convex functions $c_1,\dots,c_n$ such that $\vphi(d(\xc,\yc))=\tilde{c}(\xc-\yc)\dy\sum_{i=1}^nc_i(x_i-y_i)$ for all $\xc=(x_1,\dots,x_n),\yc=(y_1,\dots,y_n)\in\ssyn$. Let $\Sigma_0,\Sigma_1\in\zdz^n$, $\mc_0,\mc_1\in\ssyn$, and let
\begin{equation}\label{eq:map can shu}
    \begin{aligned}
        A\dy&\Sigma_0^{-\frac{1}{2}}(\Sigma_0^{\frac{1}{2}}\Sigma_1\Sigma_0^{\frac{1}{2}})^{\frac{1}{2}}\Sigma_0^{-\frac{1}{2}},
        &&\ \,\bc\dy\mc_1-A\mc_0, \\
        T(\xc)\dy& A\xc+\bc, &&\;\,\zc\dy(z_1,\dots,z_n)\dy(I-A)\xc-\bc, \\
        \Sigma\dy&(I-A)\Sigma_0(I-A),&&\mc \dy \mc_0-\mc_1, \\
        T_i(\xc)\dy&\Sigma_i^{\frac{1}{2}}\xc+\mc_i,
        &&\muc_i\dy {T_i}_\#\muc,\qquad\forall\xc\in\ssyn,i\in\{0,1\}.
    \end{aligned}
\end{equation}
Further, assume that $c_i(0)=0$ for all $i\in\{1,\dots,n\}$, and that $\muc_0$ and $\tilde{c}$ satisfy the other conditions in~\cref{thm:cost function}, then $T$ is an optimal map from $\muc_0$ to $\muc_1$ with respect to $\vphi\circ d$. By an explicit calculation,
\begin{align*}
    (\xc-\mc_0)^\zz\Sigma_0^{-1}(\xc-\mc_0)=&(\zc-\mc)^\zz\Sigma^{-1}(\zc-\mc), \\
    \implies\vphi(\Wp_{\vphi}^d(\muc_0,\muc_1))=&\int_{\ssyn}\vphi(d(\xc,A\xc+\bc))\d\muc_0(\xc) \\
    =&\det(\Sigma_0)^{-\frac{1}{2}}\int_{\ssyn}\tilde{c}((I-A)\xc-\bc)p_g(\Sigma_0^{-\frac{1}{2}}(\xc-\mc_0))\d\xc \\
    =&\det(\Sigma)^{-\frac{1}{2}}\int_{\ssyn}\tilde{c}(\zc)g((\zc-\mc)^\zz\Sigma^{-1}(\zc-\mc))\d\zc \\
    =&\det(\Sigma)^{-\frac{1}{2}}\sum_{i=1}^n\int_{\ssyn}c_i(z_i)g((\zc-\mc)^\zz\Sigma^{-1}(\zc-\mc))\d \zc.
\end{align*}
Especially if $\muc$ is the multivariate standard normal distribution on $\ssyn$, i.e.
\begin{equation*}
    g(t)\dy g_n(t)\dy(2\pi)^{-\frac{n}{2}}e^{-\frac{t}{2}},\quad\forall t\in\ssy,
\end{equation*}
then
\begin{equation}\label{eq:g_j}
    \vphi(\Wp_{\vphi}^d(\muc_0,\muc_1))=\sum_{i=1}^n (c_i*\tilde{g}_i)(m_i),
\end{equation}
where $\Sigma\dy(S_{ij})_{i,j=1}^n,\mc\dy(m_1,\dots,m_n),s_i\dy\sqrt{S_{ii}},\tilde{g}_i(t)\dy\sqrt{2\pi}s_i^{-1}e^{-t^2/(2s_i^2)}$ for all $i\in\{1,\dots,n\},t\in\ssy$.
To construct a metric $d$ on $\ssyn$ that satisfies the above conditions, we can start with metrics on $\ssy$. By~\cref{pro:generalized integral minkowski}, the following lemma can help us define metrics on the countable product of Polish spaces.
\begin{lemma}
    Let $n\in\zzs\cup\{+\infty\}$, and let $(\ml{X}_i,d_i)$ be a Polish metric space for each $i \leqslant n$ (with $i \in \zzs$). If $\vphi\in\MI$ and $\lambdac\dy(\lambda_1,\dots,\lambda_n)\in\Delta^n_+$, then
    \begin{align*}
        d_{\dc}^{\lambdac}(\xc,\yc)&\dy\vphi^{-1}\xkh{\sum_{i=1}^n\lambda_i\vphi(d_i(x_i,y_i))}, \\
        &
        \forall \xc\dy(x_1,\dots,x_n),\yc\dy(y_1,\dots,y_n)\in\prod_{i=1}^n\ml{X}_i,
    \end{align*}
    is a metric on $\prod_{i=1}^n\ml{X}_i$, where $\dc\dy(d_1,\dots,d_n)$.
\end{lemma}
\begin{remark}
    To align with the previous assumptions, we can set $d_i(x_i,y_i)\dy\eta_i|x_i-y_i|$ with $\eta_i>0$ for all $i\in\{1,\dots,n\}$.
\end{remark}

The conclusion derived above can be summarized in the following theorem.

\begin{theorem}\label{thm:W_phi general}
    Let $n\in\zzs,\etac\dy(\eta_1,\dots,\eta_n)\in\ssy^n_{>0}$, and $\lambdac\dy(\lambda_1,\dots,\lambda_n)\in\Delta^n_{+}$. Let $\vphi\in\MI\cap C^2(\ffss)$ satisfying $\int_0^{+\infty}\vphi(t)e^{-\varepsilon t^2}\d t<+\infty$ for all $\varepsilon>0$. Then
    \begin{align*}
        d_{\etac}^{\lambdac}(\xc,\yc)&\dy\vphi^{-1}\xkh{\sum_{j=1}^n\lambda_j\vphi(\eta_j|x_j-y_j|)}, \\
        &\forall \xc\dy(x_1,\dots,x_n),\yc\dy(y_1,\dots,y_n)\in\ssyn,
    \end{align*}
    is a metric on $\ssyn$.
    
    Moreover, let $\Sigma_0,\Sigma_1\in\zdz^n$ and $\mc_0,\mc_1\in\ssyn$. With the notation defined in~\eqref{eq:map can shu} and~\eqref{eq:g_j}, assume that $\lambda_j=\frac{1}{n},\eta_j=1$ for all $j\in\{1,\dots,n\}$ if $A\notin\zdjz$. Then $T$ is an optimal map from $\muc_0$ to $\muc_1$ with respect to $\vphi\circ d_{\etac}^{\lambdac}$, and
    \begin{equation*}
        \vphi(\Wp_{\vphi}^{\lambdac,\etac}(\muc_0,\muc_1))=\sum_{i=1}^n \lambda_i(\bar{\varphi}_i*\tilde{g}_i)(m_i),
    \end{equation*}
    where $\Wp_{\vphi}^{\lambdac,\etac}\dy\Wp_{\vphi}^{d_{\etac}^{\lambdac}}$, and $\bar{\varphi}_i(t)\dy\varphi(\eta_i|t|)$ for all $i\in\{1,\dots,n\},t\in\ssy$.
\end{theorem}
\begin{remark}
    According to~\cref{rm:MI examples}, for any $p\gee1$, let $\vphi(x)\dy\vphi_p(x)\dy x^p$ for all $x\gee0$, then
    \begin{align*}
        d_{\etac}^{\lambdac}(\xc,\yc)=&\zkh{\sum_{j=1}^n\lambda_j\eta_j^p|x_j-y_j|^p}^{1/p}, \\
        \vphi_p(\Wp_{\vphi_p}^{\lambdac,\etac}(\muc_0,\muc_1))
        =&\frac{2^{\frac{p}{2}}\Gamma(\frac{p+1}{2})}{\sqrt{\pi}}\sum_{i=1}^n \lambda_i\eta_i^p s_i^p\kchf{-\frac{p}{2}}{\frac{1}{2}}{-\frac{m_i^2}{2s_i^2}},
    \end{align*}
    where $\leftindex_1F_1$ is the Kummer's confluent hypergeometric function introduced in~\cite{kummerIntegralibusQuibusdamDefinitis1837}.
\end{remark}
\begin{remark}
    For the selection of $\phi$ in~\cref{rm:MI examples}, we can also consider piecewise linear functions with non-decreasing slopes. For simplicity, for any $x_0>0$, let
    \begin{equation*}
        \phi(x)\dy\begin{cases}
            u,&u\in(0,x_0) \\
            +\infty,&u\in[x_0,+\infty)
        \end{cases},
    \end{equation*}
    then
    \begin{align*}
        \vphi(x)\dy \vphi_{2,x_0}(x)\dy&x_0\int_0^xe^{\int_{x_0/2}^t\frac{1}{\phi(u)}\d u}\d t \\
        =&\begin{cases}
            x^2,&x\in[0,x_0) \\
            2x_0x-x_0^2,&x\in[x_0,+\infty)
        \end{cases}\in\MI.
    \end{align*}
    Moreover, for each $i\in\{1,\dots,n\}$, let
    \begin{align*}
        &C_{1,i}\dy\erf(\frac{x_0-\eta_im_i}{\sqrt{2}\eta_is_i}),&&C_{2,i}\dy\erf(\frac{x_0+\eta_im_i}{\sqrt{2}\eta_is_i}),\\
        &C_{3,i}\dy e^{-\frac{(x_0-\eta_im_i)^2}{2\eta_i^2s_i^2}},&&C_{4,i}\dy e^{-\frac{(x_0+\eta_im_i)^2}{2\eta_i^2s_i^2}},
    \end{align*}
    where $\erf(x)\dy\frac{2}{\sqrt{\pi}}\int_0^xe^{-t^2}\d t$, then
    \begin{align*}
        &\vphi_{2,x_0}(\Wp_{\vphi_{2,x_0}}^{\lambdac,\etac}(\muc_0,\muc_1)) \\
        =&\sum_{i=1}^n \lambda_i\Bigg\{\frac{\eta_i^2(C_{1,i}+C_{2,i})}{2}(m_i^2+s_i^2)+\frac{\eta_i^2(C_{4,i}-C_{3,i})}{\sqrt{2\pi}}m_is_i \\
        &\qquad\quad\;
        +\eta_ix_0(C_{2,i}-C_{1,i})m_i
        +\frac{\eta_ix_0(C_{3,i}+C_{4,i})}{\sqrt{2\pi}}s_i+\frac{C_{1,i}+C_{2,i}-2}{2}x_0^2\Bigg\}.
    \end{align*}
    This closed-form expression will be beneficial for the practical applications discussed below.
\end{remark}

Then, for any $\mc \in \ssyn$, let $\deltac_{\mc}$ denote the Dirac distribution centered at $\mc$. Subsequently, we can compare $\MWp_{\vphi}^{\lambdac,\etac} \mathbin{\coloneqq} \MWp_{\ml{A},\varphi}^{d_{\etac}^{\lambdac}}$ with $\Wp_{\vphi}^{\lambdac,\etac}$, where $\ml{A} \mathbin{\coloneqq} \ml{DLS}(\muc_{g_n})$.

\begin{lemma}
    Let $J_0, J_1 \in \zzs$. For each $i\in\{0,1\}$ and $j_i\in\{1,\dots,J_i\}$, let $\muc_{i,j_i}\dy\ml{N}(\mc_{i,j_i},\Sigma_{i,j_i})$ with $\Sigma_{i,j_i}\in\zdz^n,\mc_{i,j_i}\in\ssyn$, and let $\nuc_i=\sum_{j_i=1}^{J_i} \lambda_{i,j_i}\muc_{i,j_i}$ with $\lambdac_i\coloneqq(\lambda_{i,1},\dots,\lambda_{i,J_i})\in\Delta^{J_i}$. Following the notation defined in~\eqref{eq:map can shu} and~\cref{thm:W_phi general}, we have
    \begin{align*}
        \Wp_{\vphi}^{\lambdac,\etac}(\nuc_0,\nuc_1)&\lee\MWp_{\vphi}^{\lambdac,\etac}(\nuc_0,\nuc_1) \\
        &\lee\Wp_{\vphi}^{\lambdac,\etac}(\nuc_0,\nuc_1)+2\sum_{i\in\{0,1\}}\vphi^{-1}(\sum_{j=1}^{J_i}\lambda_{i,j_i}\vphi(\Wp_{\vphi}^{\lambdac,\etac}(\muc_{i,j_i},\deltac_{\mc_{i,j_i}}))).
    \end{align*}
\end{lemma}
\begin{proof}
    For each $i\in\{0,1\}$ and $j_i\in\{1,\dots,J_i\}$, let $\deltac_{i,j_i}\dy\deltac_{\mc_{i,j_i}}$ and $\tilde{\nuc}_{i}\dy\sum_{j_i=1}^{J_i} \lambda_{i,j_i}\deltac_{i,j_i}$.
    Since $\Wp_{\vphi}^{\lambdac,\etac}$ is a metric, by the triangle inequality,
    \begin{equation*}
        \Wp_{\vphi}^{\lambdac,\etac}(\nuc_0,\nuc_1)\gee\Wp_{\vphi}^{\lambdac,\etac}(\tilde{\nuc}_0,\tilde{\nuc}_1)-\Wp_{\vphi}^{\lambdac,\etac}(\nuc_0,\tilde{\nuc}_0)-\Wp_{\vphi}^{\lambdac,\etac}(\nuc_1,\tilde{\nuc}_1).
    \end{equation*}
    For each $i\in\{0,1\}$ and $j_i\in\{1,\dots,J_i\}$, let $e_{i,j_i}\coloneqq\Wp_{\vphi}^{\lambdac,\etac}(\muc_{i,j_i},\deltac_{i,j_i})$, and let $E_i\coloneqq\vphi^{-1}(\sum_{j=1}^{J_i}\lambda_{i,j_i}\vphi(e_{i,j_i}))$. Interpreting $E_i$ as the total cost of transporting each $\muc_{i,j_i}$ to $\deltac_{i,j_i}$, by~\cref{thm:two marginal equivalent}, it follows that $E_i\lee\Wp_{\vphi}^{\lambdac,\etac}(\nuc_i,\tilde{\nuc}_i)$ for each $i\in\{0,1\}$, so
    \begin{equation*}
        \Wp_{\vphi}^{\lambdac,\etac}(\nuc_0,\nuc_1)\gee\Wp_{\vphi}^{\lambdac,\etac}(\tilde{\nuc}_0,\tilde{\nuc}_1)-E_0-E_1.
    \end{equation*}
    For each $i\in\{0,1\}$ and $j_i\in\{1,\dots,J_i\}$, by the triangle inequality,
    \begin{equation*}
        \Wp_{\vphi}^{\lambdac,\etac}(\muc_{0,j_0},\muc_{1,j_1})\lee d_{\etac}^\lambdac(\mc_{0,j_0},\mc_{1,j_1})+e_{0,j_0}+e_{1,j_1}.
    \end{equation*}
    Let $\wc\coloneqq(w_{j_0,j_1})$ be an optimal discrete coupling for $\Wp_{\vphi}^{\lambdac,\etac}(\tilde{\nuc}_0,\tilde{\nuc}_1)$. By~\cref{pro:generalized minkowski},
    \begin{align*}
        \MWp_{\vphi}^{\lambdac,\etac}(\nuc_0,\nuc_1)&\lee\vphi^{-1}(\sum_{j_0,j_1}w_{j_0,j_1}\vphi(d_{\etac}^\lambdac(\mc_{0,j_0},\mc_{1,j_1})+e_{0,j_0}+e_{1,j_1})) \\
        &\lee\Wp_{\vphi}^{\lambdac,\etac}(\tilde{\nuc}_0,\tilde{\nuc}_1)+E_0+E_1.
    \end{align*}
    
    Finally, we conclude that
    \begin{equation*}
        \MWp_{\vphi}^{\lambdac,\etac}(\nuc_0,\nuc_1)
        \lee\Wp_{\vphi}^{\lambdac,\etac}(\nuc_0,\nuc_1)+2(E_0+E_1),
    \end{equation*}
    and $\Wp_{\vphi}^{\lambdac,\etac}(\nuc_0,\nuc_1)\lee\MWp_{\vphi}^{\lambdac,\etac}(\nuc_0,\nuc_1)$ by~\cref{thm:two marginal equivalent}.
\end{proof}

\begin{remark}
    For simplicity, we denote $\Wp^{\lambdac,\etac}_{\vphi_p},\MWp^{\lambdac,\etac}_{\vphi_p},\Wp^{\lambdac,\etac}_{\vphi_{2,x_0}}$, and $\MWp^{\lambdac,\etac}_{\vphi_{2,x_0}}$ as $\Wp^{\lambdac,\etac}_{p},\MWp^{\lambdac,\etac}_{p},\Wp^{\lambdac,\etac}_{2,x_0}$, and $\MWp^{\lambdac,\etac}_{2,x_0}$, respectively. Moreover,
    \begin{align*}
        &\Wp_{p}^{\lambdac,\etac}(\nuc_0,\nuc_1)\lee\MWp_{p}^{\lambdac,\etac}(\nuc_0,\nuc_1) \\
        &\quad
        \lee\Wp_{p}^{\lambdac,\etac}(\nuc_0,\nuc_1)+2^{\frac{3}{2}}\pi^{-\frac{1}{2p}}\Gamma(\frac{p+1}{2})^{\frac{1}{p}}\sum_{i\in\{0,1\}}\sum_{j=1}^{J_i}\lambda_{i,j_i}(\sum_{j=1}^n \lambda_j\eta_j^p s_{i,j_i,j}^p)^{\frac{1}{p}}, \\
        &\Wp_{2,x_0}^{\lambdac,\etac}(\nuc_0,\nuc_1)\lee\MWp_{2,x_0}^{\lambdac,\etac}(\nuc_0,\nuc_1)\lee\Wp_{2,x_0}^{\lambdac,\etac}(\nuc_0,\nuc_1)+ \\
        &\quad2\sum_{i\in\{0,1\}}\vphi_{2,x_0}^{-1}(\sum_{j=1}^{J_i}\lambda_{i,j_i}\sum_{j=1}^n \lambda_j[\eta_j^2\erf(\frac{x_0}{\sqrt{2}\eta_js_{i,j_i,j}})s_{i,j_i,j}^2+\sqrt{\frac{2}{\pi}}\eta_j x_0 s_{i,j_i,j}e^{-\frac{x_0^2}{2\eta_j^2s_{i,j_i,j}^2}}]),
    \end{align*}
    where $\Sigma_{i,j_i}\dy(S^{(i,j_i)}_{j\ell})_{j,\ell=1}^n,s_{i,j_i,j}\dy\sqrt{S^{(i,j_i)}_{jj}}$.
\end{remark}

\section{Domain Adaptation via Optimal Transport between Mixtures}\label{sec:Domain Adaptation via Optimal Transport between Mixtures}
We will present a domain adaptation framework based on the theory discussed in the previous sections and compare it with similar algorithms to demonstrate its effectiveness.

\subsection{Fitting mixture models}
In this work, we assume that the data can be effectively modeled using mixture models. For the sake of simplicity, we select the location-scale family $DLS(\muc_{g_n})$, which is actually the diagonal Gaussian mixture model. Other mixture models that satisfy~\cref{ass:Identifiability} can also be applied within the framework of this paper. Although they may involve more complex computations, these models are better suited for applications where relevant prior knowledge is available.

Let $\{\xc_i\}_{i=1}^{n_d}$ be a dataset with feature space $\ssyn$. A $DLS(\muc_{g_n})$-mixture with $J$ components has parameters $\{\lambda_j,\mc_j,\Sigma_j\}_{j=1}^J$, which can be estimated using maximum likelihood as
\begin{equation*}
    \{\lambda_j^\ast,\mc_j^\ast,\Sigma_j^\ast\}_{j=1}^J\dy\argmax_{\{\lambda_j,\mc_j,\Sigma_j\}_{j=1}^J}\sum_{i=1}^{n_d}\log P(\xc_i),
\end{equation*}
where $P\dy\sum_{j=1}^J\lambda_jP_j,P_j\dy p_{g_{\mc_j,\Sigma_j}}$. A practical method for optimizing the above equation, commonly known as the Expectation-Maximization (EM) algorithm, was proposed in~\cite{dempsterMaximumLikelihoodIncomplete1977}. In our method, we model the unlabeled domain using a single Gaussian mixture, whereas for the labeled domain, we fit a separate Gaussian mixture for each label and then integrate them. For completeness, we adopt the pseudocode from~\cite{dussonWassersteintypeMetricGeneric2026} for this approach in~\cref{alg:GMM for unlabel,alg:GMM for label}, where the function $\onehot$ converts labels to one-hot encodings.

\begin{algorithm}
\caption{Fitting a GMM for unlabeled dataset.}
\label{alg:GMM for unlabel}
\begin{algorithmic}
\Require{Unlabeled dataset $\Xc\dy\{\xc_i\}_{i=1}^{n_d}$ and the number of components $J$.}
\Ensure{The parameters of a single Gaussian mixture $\{(\lambda_j,\mc_j,\Sigma_j)\}_{j=1}^J$.}
\Procedure{EM}{$\Xc,J$}
\For{$it=1,\dots,n_{iter}$}
\State{$\pi_{i,j} \leftarrow \dfrac{\pi_{j}g_{\mc_j,\Sigma_j}(\xc_i)}{\sum_{j'}\pi_{j'}g_{\mc_{j'},\Sigma_{j'}}(\xc_i)}$;}
\Comment{Expectation Step}
\State{$\pi_{j} \leftarrow \sum_{i=1}^{n_d}\pi_{i,j}$;}
\State{$\begin{cases}
    \lambda_{j} &\!\!\!\!\leftarrow \dfrac{\pi_{k}}{n_d}; \\
    \mc_{j} &\!\!\!\!\leftarrow \dfrac{1}{\pi_{j}}\sum_{i=1}^{n_d}\pi_{i,j} \xc_i; \\
    \Sigma_{j} &\!\!\!\!\leftarrow \dfrac{1}{\pi_{j}}\sum_{i=1}^{n_d}\pi_{i,j} (\xc_{i} - \mc_{j})(\xc_{i} - \mc_{j})^\zz;
\end{cases}$}
\Comment{Maximization Step}
\EndFor
\Return{$\{(\lambda_j,\mc_j,\Sigma_j)\}_{j=1}^J$.}
\EndProcedure
\end{algorithmic}
\end{algorithm}

\begin{algorithm}
\caption{Fitting GMMs for labeled dataset.}
\label{alg:GMM for label}
\begin{algorithmic}
\Require{Labeled dataset $\hat{\Xc}\dy\{(\xc_i,y_i)\}_{i=1}^{n_d}$, the number of labels $L$ and the number of components $J$.}
\Ensure{The parameters of Gaussian mixtures $\{(\lambda_{\ell,j},\mc_{\ell,j},\Sigma_{\ell,j}, \yc_\ell)\}_{\substack{1\lee j\lee J_\ell \\
1\lee \ell\lee L}}$.}
\Procedure{ConditionalEM}{$\hat{\Xc},L,J$}
\State{$J_\ell \leftarrow \dfrac{J}{L}$;}
\Comment{Just assume that $L\mid J$}
\For{$\ell=1,\dots,L$}
\State{$\Xc_\ell \leftarrow \{\xc_i:y_i=\ell\}$;}
\Comment{Samples from $\ell$-th class}
\State{$\{\lambda_{\ell,j},\mc_{\ell,j},\Sigma_{\ell,j}\}_{j=1}^{J_\ell} \leftarrow \EM(\Xc_\ell,J_\ell)$;}
\Comment{EM on conditionals}
\State{$\yc_\ell \leftarrow \onehot(\ell)$;}
\Comment{Convert labels to one-hot encoding}
\State{$\lambda_{\ell,j} \leftarrow \lambda_{\ell,j}\cdot\dfrac{|\Xc_\ell|}{n_d}$;}
\Comment{Reset weights}
\EndFor
\Return{$\{(\lambda_{\ell,j},\mc_{\ell,j},\Sigma_{\ell,j}, \yc_\ell)\}_{\substack{1\lee j\lee J_\ell \\
1\lee \ell\lee L}}$.}
\Comment{Concatenates all parameters}
\EndProcedure
\end{algorithmic}
\end{algorithm}

\subsection{Optimal transport between Gaussian mixtures}
Referring to the ideas in~\cref{sec:Optimal transport for domain adaptation}, we first construct the optimal transport problem from the source domain $\gl^S$ to the target domain $\gl^T$. Based on the fitting in the previous section, we may assume that $\gl^S=\sum_{j_0=1}^{J_0}\lambda_{0,j_0}\muc_{0,j_0}$ and $\gl^T=\sum_{j_1=1}^{J_1}\lambda_{1,j_1}\muc_{1,j_1}$ with $\lambdac_i\dy(\lambda_1,\dots,\lambda_{J_0})\in\Delta^{J_0}$, $\muc_{i,j_i}\dy\ml{N}(\mc_{i,j_i},\Sigma_{i,j_i})$ for all $i\in\{0,1\},j_i\in\{1,\dots,J_i\}$.
According to~\cref{thm:two marginal equivalent}, we only need to solve a discrete optimal transport problem, whose cost matrix can be determined by~\cref{thm:W_phi general}. For any $\etac\dy(\eta_1,\dots,\eta_n)\in\ssy^n_{>0}$, $\lambdac\dy(\lambda_1,\dots,\lambda_n)\in\Delta^n_{+}$, and $\vphi\in\MI\cap C^2(\ffss)$, we define the cost matrix
\begin{equation*}
    C^{\vphi,\etac,\lambdac}\dy \zkh{C^{\vphi,\etac,\lambdac}_{j_0,j_1}}_{\substack{i\in\{0,1\}\\
    j_i\in\{1,\dots,J_i\}}}\dy\zkh{\vphi(\Wp_{\vphi}^{\lambdac,\etac}(\muc_{0,j_0},\muc_{1,j_1}))}_{\substack{i\in\{0,1\}\\
    j_i\in\{1,\dots,J_i\}}}.
\end{equation*}
The corresponding problem is defined by
\begin{equation*}
    \begin{aligned}
        \wc^\ast\dy\argmin_{\wc\dy(w_{j_0,j_1})_{\substack{
        1\lee j_0\lee J_0 \\
        1\lee j_1\lee J_1}}\in\Pi(\lambdac_0,\lambdac_1)}\sum_{j_0=1}^{J_0}\sum_{j_1=1}^{J_1}C^{\vphi,\etac,\lambdac}_{j_0,j_1}w_{j_0,j_1}.
    \end{aligned}
\end{equation*}
Let the total number of data points in the source and target domains be $N$, and define $J\dy\max\{J_0,J_1\}$. Compared to the optimal transport problem involving all data points, it is worth mentioning that the reduction in time complexity from $O(N^3 \log N)$ to $O(N\cdot J^3 \log J)$ and in space complexity from $O(N^2)$ to $O(N\cdot J^2)$ (where typically $J \ll N$) represents a transition from a computationally intensive task to a real-time operation. This advancement enables Wasserstein-based clustering and distribution alignment in large-scale machine learning systems.

For simplicity, in the subsequent numerical experiments, let $\vphi=\vphi_p$ with $p\gee1$,
and let $\lambda_i=\frac{1}{n},\eta_i=1$ for all $i\in\{1,\dots,n\}$.

\subsection{Estimation methods}
Based on the optimal matching solution $\wc^\ast$ obtained in the previous section, we aim to establish a classifier $h$ in the target domain or an effective mapping $T:\ssyn\to\ssyn$ that aligns $\gl^S$ with $\gl^T$. The literature posits four primary methodologies for this estimation, as outlined below.
\subsubsection{Maximum a posteriori classifier \texorpdfstring{$h_{MAP}$}{}}
In the GMM-based Optimal Transport (GMM-OT) framework, the matching matrix $\wc^\ast\in\ssy^{L\times L}$ is interpreted as a joint probability distribution over the latent component spaces of the source and target domains. Specifically, for the $j_0$-th source component and the $j_1$-th target component, the entry $w^\ast_{j_0,j_1}$ represents the joint probability $\Prob(J^S=j_0,J^T=j_1)$, where $J^S$ and $J^T$ are categorical random variables that mapping each data point to its latent generating component in the source and target GMMs, respectively.

Let $Y^S$ and $Y^T$ be the random variables that assign each data point to its label in the source and target domains, respectively, and let $X^T\sim\gl^T$. Similar to the covariate shift hypothesis~\cite{sugiyamaCovariateShiftAdaptation2007}, we assume that
\begin{align*}
    &\Prob(Y^T|J^S=j_0,J^T=j_1)=\Prob(Y^S|J^S=j_0) \\
   \implies&\Prob(Y^T|J^T=j_1)=\sum_{j_0=1}^{J_0}\Prob(Y^S|J^S=j_0)\Prob(J^S=j_0|J^T=j_1),
\end{align*}
for all $j_1\in\{1,\dots,J_1\}$.
Then, using label propagation~\cite{redkoOptimalTransportMultisource2019}, the estimated label of the $j_1$-th target component can be defined as follows:
\begin{equation*}
    \hat{\yc}^T_{j_1}\dy\frac{1}{\lambda_{1,j_1}}\sum_{j_0=1}^{J_0}w^\ast_{j_0,j_1}\yc^S_{j_0},
\end{equation*}
where $\yc^S_{j_0}$ is the label of the $j_0$-th source component. Next, following~\cite{montesumaOptimalTransportDomain2025}, we can use maximum a posteriori estimation to establish a classifier in the target domain:
\begin{align*}
    h_{MAP}\dy&\argmax_{y\in\{1,\dots,L\}}\Prob(Y^T=y|X^T=\xc^T) \\
    =&\argmax_{y\in\{1,\dots,L\}}\sum_{j_1=1}^{J_1}\Prob(Y^T=y|X^T=\xc^T,J^T=j_1)\Prob(J^T=j_1|X^T=\xc^T) \\
    =&\argmax_{y\in\{1,\dots,L\}}\sum_{j_1=1}^{J_1}\frac{\lambda_{1,j_1}g_{\mc_{1,j_1},\Sigma_{1,j_1}}(\xc^T)}{\sum_{j'_1=1}^{J_1}\lambda_{1,j'_1}g_{\mc_{1,j'_1},\Sigma_{1,j'_1}}(\xc^T)}\hat{\yc}^T_{j_1},
\end{align*}
where the last equality holds because classes and components are structurally congruent within the data points under the usual settings.
\subsubsection{Conditional mean mapping \texorpdfstring{$T_{mean}$}{}}
The mean mapping strategy, as discussed by~\cite[Section 6.3]{delonWassersteinTypeDistanceSpace2020}, is derived from the barycentric projection. For a given source sample $\xc^S$, the transported point is defined as the conditional expectation under the optimal transport plan $\gammac^\ast$:
\begin{equation*}
    T_{mean}(\xc^S)\dy\qw_{\xc^T\sim\gammac^\ast(\,\cdot\,|\xc^S)}[x^T].
\end{equation*}
While $T_{mean}$ provides a deterministic and continuous transformation, it frequently fails to maintain the global structural properties of the target distribution. Specifically, the mapping may lead to a "shrinkage" effect where the transported points converge toward the global mean, thus failing to accurately populate the distinct modes of $\gl^T$~\cite{delonWassersteinTypeDistanceSpace2020}.
\subsubsection{Stochastic mapping \texorpdfstring{$T_{rand}$}{}}
To circumvent the limitations of mean-based methods, $T_{rand}$ introduces a probabilistic assignment mechanism~\cite[Section 6.3]{delonWassersteinTypeDistanceSpace2020}. A source point $\xc^S$ is mapped to the $j_1$-th target component via the affine map $T_{j_0,j_1}$ with a probability $P_{j_0,j_1}(\xc^S)$ defined as
\begin{equation*}
    P_{j_0,j_1}(\xc^S)\dy \frac{w^\ast_{j_0,j_1}g_{\mc_{0,j_0},\Sigma_{0,j_0}}(\xc^S)}{\sum_{j_0'=1}^{J_0}\lambda_{0,j'_0}g_{\mc_{0,j_0'},\Sigma_{0,j_0'}}(\xc^S)}.
\end{equation*}
Theoretically, $T_{rand}$ ensures that the push-forward distribution ${T_{rand}}_{\#}\gl^S$ exactly recovers $\gl^T$. However, from a practical optimization perspective, the stochastic nature of the index sampling $(j_0,j_1)$ introduces significant variance and irregularity in the mapping, which may degrade the stability of subsequent classification models.
\subsubsection{Importance weighted mapping \texorpdfstring{$T_{weight}$}{}}
The $T_{weight}$ strategy, proposed by~\cite[Section 4.2]{montesumaOptimalTransportDomain2025}, regularizes the transport process by combining deterministic component identification with weighted mass distribution. The procedure is formalized as follows:
\begin{list}{\arabic{enumi}.}{\usecounter{enumi}
\leftmargin=2em
\labelwidth=2em
\labelsep=0.5em}
  \item Maximum a posterior selection: the label of the source component which most likely to have generated the observation $\xc^S$ is determined by
  \begin{equation*}
      \hat{j}_0\dy\argmax_{j_0\in\{1,\dots,J_0\}}\frac{\lambda_{0,j_0}g_{\mc_{0,j_0},\Sigma_{0,j_0}}(\xc^S)}{\sum_{j_0'=1}^{J_0}\lambda_{0,j'_0}g_{\mc_{0,j_0'},\Sigma_{0,j_0'}}(\xc^S)}.
  \end{equation*}
  \item Multicomponent Transport: the point is mapped to all the $j_1$-th target components satisfying $w^\ast_{\hat{j}_0,j_1}\gee\tau$, where $\tau$ is a sparsity threshold.
  \item Weight assignment: each generated target point $T_{\hat{j}_0,j_1}(\xc^S)$ is assigned an importance weight $w^\ast_{\hat{j}_0,j_1}$.
\end{list}
$T_{weight}$ yields a piece-wise affine transformation that inherits the geometric properties of the GMM components. By enforcing a form of group-sparsity (where samples from the same source cluster are transported collectively), this method preserves class-discriminative features.
\subsection{Experiments}
To evaluate the scalability and robustness of our proposed framework, we conduct experiments on the VisDA-C benchmark~\cite{pengVisDASynthetictoRealBenchmark2018}, a large-scale DA dataset consisting of $152,397$ source samples and $55,388$ target samples. Following previous research~\cite{elhamriHierarchicalOptimalTransport2022}, we first pre-train ResNet~\cite{heDeepResidualLearning2016} or ViT~\cite{dosovitskiyImageWorth16x162021} networks on the source domain, then perform shallow DA within the derived feature space. This benchmark presents a significant computational challenge for traditional empirical OT methods. Specifically, solving a discrete OT problem of this magnitude would require a linear program with approximately $8.4\times10^9$ variables. Moreover, implementing a barycentric mapping would necessitate storing a transport plan $\gammac$ with a comparable number of floating-point coefficients, resulting in a memory footprint of roughly $270$ GB\,---\,exceeding the capacity of standard computing environments.

Referencing~\cite{montesumaOptimalTransportDomain2025}, we compare our LSMM-OTDA strategies with other OT methods for DA, using the notation introduced therein. To facilitate a comparative analysis under these constraints, empirical OT baselines were executed on a reduced sub-sample of $n_S=n_T=15,000$. In contrast, the parametric nature of LSMM-OTDA allows processing the entire dataset, highlighting the intrinsic efficiency of adopting compact distributional representations. For feature extraction, we choose ResNet-50, ResNet-101, and ViT-B/16. The quantitative results are summarized in~\cref{tab:visda_results}.

\begin{table}[htbp]
    \centering
    \caption{Comparison of the performance of domain adaptation algorithms across various feature extractors pre-trained on source domain data. Accuracy (in \%) and the difference $\Delta$ relative to the source-only baseline are reported (Part of the data was adapted from~\cite{montesumaOptimalTransportDomain2025}).}
    \label{tab:visda_results}
    \vspace{0.5em}
    \small
    \setlength{\tabcolsep}{1.5pt}
    \setlength{\heavyrulewidth}{1.5pt}
    \begin{tabular}{l c c c c}
        \toprule
        \textbf{Algorithm} & \textbf{Parameter} & \textbf{ResNet 50} & \textbf{ResNet 101} & \textbf{ViT-B/16} \\
        \midrule
        Source-Only & - & 47.93 & 53.90 & 56.70 \\
        \midrule
        OTDA\textsubscript{EMD}         & - & 53.69 \posdelta{5.76}  & 57.42 \posdelta{3.52}  & 63.25 \posdelta{6.55} \\
        OTDA\textsubscript{Sinkhorn}    & - & 53.02 \posdelta{5.09}  & 10.54 \negdelta{43.36} & 66.75 \posdelta{10.05} \\
        OTDA\textsubscript{Affine}      & - & 7.46 \negdelta{40.47}  & 11.82 \negdelta{42.08} & 6.75 \negdelta{49.95} \\
        OTDA\textsubscript{Affine-Diag} & - & 51.41 \posdelta{3.48}  & 56.91 \posdelta{3.01}  & 59.94 \posdelta{3.24} \\
        HOTDA                           & - & 47.51 \negdelta{0.42}  & 47.64 \negdelta{6.26}  & 62.55 \posdelta{5.85} \\
        \midrule
        LSMM-OTDA\textsubscript{MAP}     & $p=2.0$ & 53.94 \posdelta{6.01}  & 49.48 \negdelta{4.42} & 66.76 \posdelta{10.06} \\
        LSMM-OTDA\textsubscript{weight}  & $p=2.0$ & 57.90 \posdelta{9.97} & 52.47 \negdelta{1.43}  & 70.19 \posdelta{13.49} \\
        \midrule
        LSMM-OTDA\textsubscript{MAP}     & $p=1.0$ &  52.83 \posdelta{4.90}  &  55.16 \posdelta{1.26} & \textbf{72.89} \posdelta{16.19} \\
        LSMM-OTDA\textsubscript{weight}  & $p=1.0$ &  \textbf{58.47} \posdelta{10.54}  &  \textbf{59.66} \posdelta{5.76} & 70.70 \posdelta{14.00} \\
        \bottomrule
    \end{tabular}
\end{table}

Notably, we observe a strong positive correlation between the training accuracy and testing accuracy of the LSMM-OTDA\textsubscript{weight} method (see~\cref{fig:Zheng Xiang Guan}). This finding suggests that training performance serves as a reliable proxy for evaluating model generalization. Consequently, the superior parameter $p$ can be determined a priori based on the training set, since higher training accuracy typically indicates better test performance.

\begin{figure}[htbp]
    \centering
    \begin{tikzpicture}
        \begin{axis}[
            width=0.9\textwidth,
            height=0.65\textwidth,
            xlabel={Final Training Accuracy (\%)},
            ylabel={Testing Accuracy (\%)},
            title={Correlation between Training and Testing Accuracy of LSMM-OTDA\textsubscript{weight}},
            grid=major,
            xmin=68.0, xmax=72.0,
            ymin=68.0, ymax=72.0,
            legend pos=north west,
            tick label style={font=\small},
            label style={font=\small}
        ]

            \pgfplotstableread{
                p       Last        acc
                2.2     68.517250   68.529284
                1.1     68.566668   68.578031
                1.2     68.882340   68.890373
                2.3     69.039786   69.052863
                2.5     69.379693   69.392287
                2.4     69.391614   69.404925
                2.6     69.469905   69.482559
                2.9     69.658256   69.670326
                3.0     69.825345   69.838232
                2.8     70.131367   70.145158
                2.0     70.175773   70.190294
                2.7     70.189416   70.201127
                2.1     70.511049   70.527912
                1.3     70.602208   70.612768
                1.4     70.668566   70.677764
                1.0     70.695949   70.703040
                1.6     70.913243   70.923305
                1.7     71.031231   71.040659
                1.5     71.133369   71.143569
                1.9     71.316928   71.322308
                1.8     71.313792   71.324114
            }\datatable

            \addplot[
                blue, 
                domain=68.2:71.5, 
                samples=2, 
                thick,
                dashed
            ] {0.998932*x + 0.086118}; 
            \addlegendentry{Trend Line ($R^2 > 0.9999$)}

            \addplot[
                only marks,
                mark=*,
                mark size=1.8pt,
                red!70!black,
            ] table [x=Last, y=acc] {\datatable};
            \addlegendentry{Experimental Results}

            \node[coordinate, pin={[pin distance=0.25cm, pin edge={black!60, thin}] 135:{\scriptsize $p=2.2$}}] at (axis cs:68.517250,68.529284) {};
            \node[coordinate, pin={[pin distance=0.25cm, pin edge={black!60, thin}] 315:{\scriptsize $p=2.3$}}] at (axis cs:69.039786,69.052863) {};
            \node[coordinate, pin={[pin distance=0.25cm, pin edge={black!60, thin}] 135:{\scriptsize $p=2.5$}}] at (axis cs:69.379693,69.392287) {};
            \node[coordinate, pin={[pin distance=0.25cm, pin edge={black!60, thin}] 315:{\scriptsize $p=3.0$}}] at (axis cs:69.825345,69.838232) {};
            \node[coordinate, pin={[pin distance=0.25cm, pin edge={black!60, thin}] 135:{\scriptsize $p=2$}}] at (axis cs:70.175773,70.190294) {};
            \node[coordinate, pin={[pin distance=0.25cm, pin edge={black!60, thin}] 315:{\scriptsize $p=1.0$}}] at (axis cs:70.695949,70.703040) {};
            \node[coordinate, pin={[pin distance=0.25cm, pin edge={black!60, thin}] 135:{\scriptsize $p=1.8$}}] at (axis cs:71.313792,71.324114) {};

        \end{axis}
    \end{tikzpicture}
    \caption{Correlation analysis between the final training accuracy and testing accuracy under the pre-trained ViT-B/16 weights. The dashed blue line serves as a reference indicating the strong linear consistency between training and testing performance.}
    \label{fig:Zheng Xiang Guan}
\end{figure}
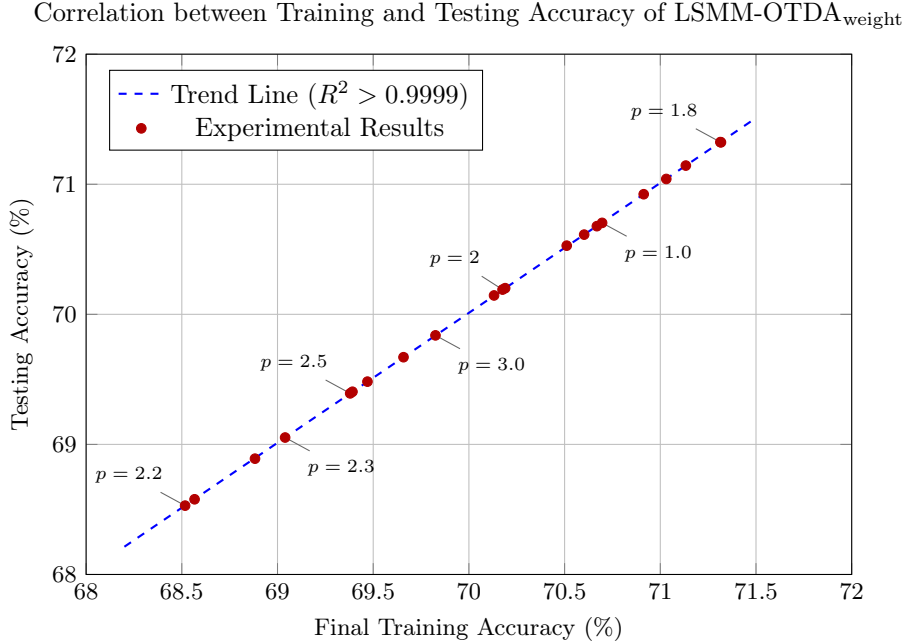

\section{Conclusion}

In this paper, we establish a general mathematical framework for optimal transport (OT) between identifiable finite location-scale mixture models. Our theoretical contributions are threefold. First, by defining a specific function class grounded in generalized Minkowski-type inequalities, we extend traditional Wasserstein-type metrics and barycenters to these mixture models. Second, we provide a characterization of translation-invariant cost functions that admit affine OT maps. Specifically, we prove that the optimality of symmetric positive definite affine maps requires the cost function to be separable, and that the presence of any cross-coupling further reduces the cost to a quadratic form. Third, we prove that restricting joint couplings to specific mixture structures reduces the continuous multimarginal OT problem to a tractable discrete formulation over mixture components, accompanied by bounding inequalities against their original continuous counterparts.

From a computational perspective, these theoretical foundations translate into scalable algorithms for machine learning applications. By shifting the transport paradigm from pointwise empirical matching to component-level matching, we propose an efficient domain adaptation framework instantiated via Gaussian mixture models. This approach decouples the core OT computation from individual data points, reducing the computational complexity to scale linearly with the sample size. Empirical evaluations on the VisDA-C benchmark demonstrate that our method achieves significant reductions in both computational overhead and memory footprint, while maintaining competitive classification accuracy compared to existing empirical OT approaches.

Looking forward, this generalized framework opens several avenues for future research. First, while our current theoretical analysis relies on the identifiability of mixture models, extending these generalized metrics to non-identifiable or more complex distributions remains an open mathematical challenge. Second, although our algorithmic design and empirical evaluations focus on the foundational two-marginal domain adaptation setting, a natural progression is to develop and evaluate practical algorithms for multi-source domain adaptation, thereby leveraging the multimarginal barycenter theory established herein. Finally, integrating this efficient transport mechanism into the training objectives of deep generative models may offer a principled approach to refining latent space alignment in large-scale representation learning.

\section*{Acknowledgments}
This research is supported by National Key R\&D Program of China (2024YFA1012401), the Science and Technology Commission of Shanghai Municipality (23JC1400501), and Natural Science Foundation of China (12241103).

%\let\url\bibcustomurl
%\bibliographystyle{siamplain}
%\bibliography{202606}

\end{document}